\pdfoutput=1
\documentclass[toc]{mathprint}
\usepackage{macros}

\title[Multifractal Gatzouras--Lalley measures]{Multifractal analysis of Gatzouras--Lalley measures}

\author[Jordan]
  {Thomas Jordan}
  {
      School of Mathematics,
      Fry Building,
      Woodland Road,
      Bristol,
      BS8 1UG,
      United Kingdom
  }
  {thomas.jordan@bristol.ac.uk}

\author[Kolossváry]
  {István Kolossváry}
  {
      HUN-REN Alfréd Rényi Institute of Mathematics,
      1053 Budapest,
      Reáltanoda u.\ 13--15,
      Hungary
  }
  {istvanko@renyi.hu}

\author[Rutar]
  {Alex Rutar}
  {
      Department of Mathematics,
      The University of British Columbia,
      1984 Mathematics Road,
      Vancouver BC V6T 1Z2,
      Canada
  }
  {alex@rutar.org}

\begin{document}
\begin{abstract}
    We study the multifractal analysis of self-affine measures supported on planar carpets of type Gatzouras--Lalley.
    We work primarily with Olsen's Hausdorff-dimensional variant of the $L^q$ spectrum, which we denote by $\vartheta(q)$.
    We prove a variational formula for $\vartheta$, from which it follows that multifractal formalism holds at all values $q$ for which $\vartheta(q)$ is differentiable.

    However, in general, $\vartheta$ need not be differentiable and the multifractal formalism may fail.
    In fact, the multifractal spectrum may be non-differentiable and have jump discontinuities in the interior of its support, and there may be multiple phase transitions at both negative and positive values of $q$.
\end{abstract}

\section{Introduction}
A finite family of maps $\{f_i\}_{i\in\mathcal{I}}$ is a \emph{self-affine iterated function system (IFS)} if $f_i(x) = A_i x + t_i$ where $A_i$ is a strictly contracting linear map from $\R^d$ to $\R^d$, and $t_i\in\R^d$ is some translation.
(In this paper, we will take $d=2$.)
Given a probability measure $\bm{p} = (p_i)_{i\in\mathcal{I}}$ on $\mathcal{I}$, there exists a unique compactly supported Borel probability measure $\mu_{\bm{p}}$ satisfying the invariance relation
\begin{equation*}
    \mu_{\bm{p}}=\sum_{i\in\mathcal{I}}p_i \mu_{\bm{p}}\circ f_i^{-1}.
\end{equation*}
We call $\mu_{\bm{p}}$ a \emph{self-affine measure} and let $\Lambda = \supp\mu_{\bm{p}}$ denote the associated \emph{self-affine set}.
Without loss of generality, we will assume that $\bm{p}$ is strictly positive.

In this paper, we study the multifractal properties of self-affine measures supported on a family of planar self-affine carpets.
We refer to such measures as \emph{Gatzouras--Lalley measures}.
We defer a formal definition to \cref{ss:GL}, but in words the most notable features are that the matrices are diagonal with horizontal contractions $\ctr_{i,1}$ and vertical contractions $\ctr_{i,2}$ for $i\in\mathcal{I}$, the contractions satisfy $0 < \ctr_{i,2} < \ctr_{i,1} < 1$, and the maps are arranged into columns.
See \cref{f:carpets} for a representative example of the maps in such an IFS as well as the corresponding invariant set.

\subsection{Background on multifractal analysis}
Suppose $\mu$ is a Borel probability measure.
The \emph{local dimension} of $\mu$, denoted by $\dim_{\loc}(\mu, x)$, is given (when the limit exists) by
\begin{equation*}
    \dim_{\loc}(\mu,x)=\lim_{r\to 0}\frac{\log \mu\bigl(B(x,r)\bigr)}{\log r}.
\end{equation*}
For instance, for general self-affine measures, it is known that $\dim_{\loc}(\mu, x)$ exists and is constant for $\mu$-a.e.\ $x$; see \cite{zbl:1564.28002}.
However, the set of points for which the local dimension exists but takes a different value can still be large.
Therefore, a natural object of interest is the \emph{multifractal spectrum}:
\begin{equation*}
    f_\mu(\alpha)\coloneqq \dimH\{x\in\supp\mu:\dim_{\loc}(\mu,x)=\alpha\}.
\end{equation*}
Here and elsewhere, the Hausdorff dimension of the empty set is $-\infty$.
We emphasize here that we actually require the local dimension to exist, rather than take the lower local dimension.

Commonly associated with the multifractal spectrum of $\mu$ is the \emph{(coarse) $L^q$-spectrum} of $\mu$, defined for $q\in\R$ by
\begin{equation*}
    \tau_\mu(q) = \liminf_{r \to 0}\frac{\log \sup \sum_i \mu\bigl(B(x_i, r)\bigr)^q}{\log r}.
\end{equation*}
Here, the supremum is taken over all packings $\{B(x_i, r)\}_i$ with $x_i\in\supp\mu$.
The function $\tau_\mu$ is an increasing concave function.
A heuristic relationship between the functions $f_\mu$ and $\tau_\mu$, called the \emph{multifractal formalism}, was first proposed in the physics literature \cite{zbl:0538.58026}:
\begin{equation}\label{e:mf}
    f_\mu(\alpha) = \tau_\mu^*(\alpha)\coloneqq\inf_{q\in\R} \{q\alpha - \tau_\mu(q)\}.
\end{equation}
Here, $\tau_\mu^*$ is the \emph{concave conjugate} of $\tau_\mu$.
More generally, one hopes to understand the function $f_\mu$ by studying the function $\tau_\mu$ and using duality.

In general, however, the $L^q$-spectrum is not the most natural dual object to the multifractal spectrum.
For instance, it is easy to see that $\tau_{\mu}(0) = -\dimuB \supp \mu$, and for self-affine measures (in particular, for those which we study in this paper) it is often the case that $\dimH\supp\mu < \dimuB\supp \mu$ (for the earliest examples, see \cite{local:BedfordThesis,zbl:0539.28003}).
In this case, \cref{e:mf} cannot hold.

Instead, we consider a variant of the $L^q$-spectrum first introduced by Olsen \cite{zbl:0841.28012} which is in general a more plausible candidate as a dual to the multifractal spectrum.
We will denote this function by $\vartheta_\mu$, and refer to it as the \emph{fine $L^q$-spectrum}.
The precise definition is rather technical, so we defer it to \cref{ss:lq}.
In short, the definition of $\vartheta_\mu$ is to the definition of $\tau_\mu$ as the definition of Hausdorff dimension is to the definition of upper box dimension.
For example, $\vartheta_\mu(0) = -\dimH\supp\mu$.

\subsection{Multifractal analysis of self-affine measures}
\begin{figure}[t]
    \centering
    \begin{subcaptionblock}{.47\textwidth}
        \centering
        \begin{tikzpicture}[scale=5]
    \begin{scope}[font=\tiny]
        \node[below] at (0,0) {$0$};
        \node[below] at (1,0) {$1$};
        \node[left] at (0,0) {$0$};
        \node[left] at (0,1) {$1$};
    \end{scope}
    \draw (0,0) rectangle (1,1);

    \draw[dashed] (1/3,0) -- (1/3,1);
    \draw[dashed] (1/2,0) -- (1/2,1);

    \begin{scope}[thick]
        \draw[fill=gray!10] (0.0,0.0) rectangle (1/3,1/4);
        \draw[fill=gray!10] (0.0,1/2) rectangle (1/3,1/2+1/10);

        \draw[fill=gray!10] (1/2,1/4-1/7) rectangle (1,1/4+1/3-1/7);
        \draw[fill=gray!10] (1/2,3/4) rectangle (1,1);
    \end{scope}
\end{tikzpicture}
        \caption{Maps}
        \label{f:carpet-maps}
    \end{subcaptionblock}%
    \begin{subcaptionblock}{.47\textwidth}
        \centering
        \input{figures/carpet_attractor/fig}
        \caption{Attractor}\label{sf:bar}
        \label{f:carpet-attractor}
    \end{subcaptionblock}%
    \caption{Generating maps and attractor associated with a Gatzouras--Lalley IFS.}
    \label{f:carpets}
\end{figure}
In the special case that the maps $f_i$ are self-similar (that is, every $A_i$ is a scalar multiple of an orthogonal transformation) and moreover satisfy the open set condition (that is, there is an open set $U$ so that $f_i(U)\subset U$ for all $i$ and $f_i(U) \cap f_j(U) = \varnothing$ for all $i\neq j$), it is well-known that
\begin{equation*}
    \sum_{i\in\mathcal{I}}p_i^q r_i^{-\tau_{\mu_{\bm{p}}}(q)} = 1
\end{equation*}
where each $f_i$ has similarity ratio $r_i\in(0,1)$.
In particular, \emph{$\tau_{\mu_{\bm{p}}}$ is analytic and moreover the multifractal formalism holds}.
For exposition concerning this latter fact and its history, see \cite{arxiv:2312.08974,zbl:1543.28001,zbl:0869.28003,zbl:0873.28003,zbl:0763.58018,zbl:0664.58022}.
In the self-similar case with separation, it is also known that $\vartheta_\mu=\tau_\mu$ \cite{zbl:0841.28012}.
Understanding the overlapping self-similar case is an area of active research; see for instance \cite{zbl:1426.11079,zbl:1184.28009,zbl:1561.28096} and the many references therein.

In contrast, much less is known about multifractal analysis for general self-affine measures.
Firstly, we note that there are some results for Birkhoff averages \cite{zbl:1230.37034,zbl:1150.28004} (these results are quite different than for local dimensions studied here).
For self-affine IFSs either with typical translations or under certain assumptions about projections in Furstenberg directions and assumptions on the matrices, there are partial results; see \cite{zbl:1280.28010,zbmath:8179359,zbl:0946.28004,zbl:1379.28008}.
In the above results, the Hausdorff and box dimensions are the same so it is natural to work instead with the ``usual'' $L^q$-spectrum $\tau_\mu$.
In contrast, for self-affine carpets (such as Bedford--McMullen carpets, or the carpets considered here), the Hausdorff and box dimensions are typically distinct so it is natural to study the fine $L^q$-spectrum instead.
For Bedford--McMullen carpets, there is an explicit closed-form expression for $\vartheta_\mu$ and it is known that $\vartheta_\mu$ is analytic, and $f_\mu = \vartheta_\mu^*$ (that is, the \emph{fine multifractal formalism} holds).
This was proven by Olsen \cite{zbl:0955.28004} (building on earlier work of King \cite{zbl:0845.28007}); see also \cite{zbl:1206.28012,zbl:1206.82004}.
We are not aware of any other results concerning the fine $L^q$ spectrum of self-affine measures.

In this paper, we study the fine multifractal spectrum of measures supported on Gatzouras--Lalley carpets.
At first glance, Gatzouras--Lalley carpets seem rather similar to Bedford--McMullen carpets, with the only extra flexibility being that the vertical and horizontal contraction ratios can vary between maps (as long as the column-structure remains preserved).
However, to disabuse the reader of such a belief, we note two important results:
\begin{enumerate}
    \item Gatzouras--Lalley carpets can support multiple Bernoulli measures of maximal dimension \cite{zbl:1228.37025}.
    \item For the natural generalization of Gatzouras--Lalley carpets to $\R^3$, it can happen that there is \emph{no} ergodic measure of maximal dimension \cite{zbl:1387.37026}.
\end{enumerate}
This is in contrast to the case for the box dimension, where explicit formulas are known for the box dimension \cite{zbl:0757.28011} and for the $L^q$ spectrum \cite{zbl:1091.28005}, even for the natural higher-dimensional generalizations \cite{zbl:1549.37013}.
Also compare, for instance, the results concerning the Assouad spectrum of Gatzouras--Lalley carpets from \cite{zbl:1587.28018} with the simple formula for Bedford--McMullen carpets proven in \cite{zbl:1407.28002}.

As a consequence of our main results, we contribute a third result to this list:
\begin{enumerate}[resume]
    \item Gatzouras--Lalley carpets can support Bernoulli measures $\mu$ for which the fine multifractal formalism can fail: $\vartheta_\mu$ is non-differentiable and there is a number $\alpha>0$ such that $0<f_\mu(\alpha) < \vartheta_\mu^*(\alpha)$.
        In fact, the multifractal spectrum can be discontinuous in the interior of its support.
\end{enumerate}

We now proceed to state our main results precisely.
The notation and definitions required to state the main results can be found in \cref{ss:GL}, \cref{ss:ent}, and \cref{ss:lq}, where they are introduced in context along with other preliminary material.

\subsection{Main results}
Throughout this section and later, when $\mu=\mu_{\bm{p}}$ is a self-affine measure, we will use the shorthand
\begin{equation*}
    f_{\bm{p}}=f_{\mu_{\bm{p}}},\qquad \vartheta_{\bm{p}}=\vartheta_{\mu_{\bm{p}}},\qquad \tau_{\bm{p}} = \tau_{\mu_{\bm{p}}}.
\end{equation*}
Here, $f_{\bm{p}}$ denotes the multifractal spectrum, $\vartheta_{\bm{p}}$ denotes the fine $L^q$-spectrum, and $\tau_{\bm{p}}$ denotes the (usual) coarse $L^q$-spectrum.
Without loss of generality, we may assume that $\bm{p}$ is fully supported.
We also assume that the IFS has separation of principal projections (see \cref{d:spp}).

Our main contribution is to prove a variational formula for $\vartheta_{\bm{p}}(q)$ for all $q\in\R$.
In order to state the variational formula, we require a bit of notation.

For $\bm{w}, \bm{p}\in\mathcal{P}$ (where $\mathcal{P}$ denotes the probability simplex on $\mathcal{I}$), define
\begin{equation}\label{e:ly-formulas}
    \begin{aligned}
        v(\bm{w}) &= \frac{H(\eta(\bm{w}))}{\chi_1(\eta(\bm{w}))}+\frac{H(\bm{w})-H(\eta(\bm{w}))}{\chi_2(\bm{w})},\\
        u_{\bm{p}}(\bm{w}) &= \frac{H(\eta(\bm{w}), \eta(\bm{p}))}{\chi_1(\eta(\bm{w}))}+\frac{\CH{\bm{w}}{\bm{p}}-\CH{\eta(\bm{w})}{\eta(\bm{p})}}{\chi_2(\bm{w})}.
    \end{aligned}
\end{equation}
See \cref{ss:ent} for the definitions of entropy $H$ and Lyapunov exponents $\chi_k$ and \cref{ss:GL} for the precise definition of the projection $\eta$ from the space of maps to the space of columns.
The quantity $v(\bm{w})$ is just the Hausdorff dimension of $\mu_{\bm{w}}$, and the quantity $u_{\bm{p}}(\bm{w})$ is the $\mu_{\bm{w}}$-a.e.\ value of the local dimension of $\mu_{\bm{p}}$.

We also introduce a \emph{fibred $L^q$-spectrum}.
Let $\Psi\colon\eta(\mathcal{P})\times\R\times\R\to\R$ be given by
\begin{equation*}
    \Psi(\bm{v},q,s)=-\sum_{\ell\in\eta(\mathcal{I})}\bm{v}_\ell\log\Bigl(\sum_{j\in\eta^{-1}(\ell)}p_j^q\ctr_{j,2}^{-s}\Bigr).
\end{equation*}
Clearly, $\Psi$ is continuous and strictly decreasing in $s$.
Moreover,
\begin{equation*}
    \lim_{s\to -\infty}\Psi(\bm{v},q,s)=\infty
    \qquad\text{and}\qquad
    \lim_{s\to\infty}\Psi(\bm{v},q,s)=-\infty.
\end{equation*}
In particular, this implies for each $\bm{p}\in\mathcal{P}$ that there is a unique continuous function $T_{\bm{p}}\colon\eta(\mathcal{P})\times\R\to\R$ satisfying
\begin{equation*}
    \Psi\bigl(\bm{v},q,T_{\bm{p}}(\bm{v},q)\bigr)=q\CH{\bm{v}}{\eta(\bm{p})}.
\end{equation*}
We prove the following formula.
\begin{itheorem}\label{it:main}
    Let $\mu_{\bm{p}}$ be a Gatzouras--Lalley measure with separation of principal projections and let $q\in\R$.
    Then
    \begin{align*}
        \vartheta_{\bm{p}}(q)
        &=\min_{\bm{w}\in\mathcal{P}}\left\{q u_{\bm{p}}(\bm{w}) - v(\bm{w})\right\} \\
        &=\min_{\bm{v}\in\eta(\mathcal{P})}\Bigl\{\frac{q\CH{\bm{v}}{\eta(\bm{p})}-H(\bm{v})}{\chi_1(\bm{v})}+T_{\bm{p}}(\bm{v},q)\Bigr\}.
    \end{align*}
\end{itheorem}
Here are some relevant comments on this theorem.
\begin{enumerate}
    \item When $q = 0$, the first line is precisely the variational formula for the Hausdorff dimension of $\Lambda$, as established in \cite{zbl:0757.28011}.
        However, even in the case $q=0$, our proof is different.
    \item The second formula is new even in the case $q = 0$.
        Heuristically, quasi-concavity (see \cref{ss:qc}) allows us to reduce dimensionality of the optimization ``by one'': in $\R$, there is a closed form expression; and in $\R^2$, we still have a variational formula.
        However, starting in $\R^3$, there is no more variational formula using ergodic measures \cite{zbl:1387.37026}.
    \item In the special case that $\chi_1(\bm{w})/\chi_2(\bm{w})$ is constant, our framework also allows us to prove that $\vartheta_{\bm{p}}$ is differentiable.
        More generally, this argument extends to higher dimensions, under the assumption that the joint ratios of Lyapunov exponents are constant.
        This is a mild generalization of a result due to Olsen \cite{zbl:0955.28004}.
\end{enumerate}
Using our main result, we obtain consequences for the multifractal formalism.
We first recall the following well-known lower bound:
\begin{equation*}
    f_{\bm{p}}(\alpha) \geq \max_{\bm{w}\in\mathcal{P}}\left\{v(\bm{w}):u_{\bm{p}}(\bm{w}) = \alpha\right\}.
\end{equation*}
Since this lower bound is dual (in the sense explained in \cref{ss:optimisation}) to the variational formula for $\vartheta_{\bm{p}}$, we immediately obtain the fine multifractal formalism whenever $\vartheta_{\bm{p}}$ is differentiable.
\begin{icorollary}\label{ic:diff}
    Let $\mu_{\bm{p}}$ be a Gatzouras--Lalley measure with separation of principal projections.
    Then, $f_{\bm{p}}(\alpha) = \vartheta_{\bm{p}}^*(\alpha)$ for all values $\alpha=\vartheta_{\bm{p}}'(q)$ where $q\in\R$ and $\vartheta_{\bm{p}}'(q)$ exists.
\end{icorollary}
\begin{remark}
    The above result also holds at the endpoints, by defining $\vartheta_{\bm{p}}'(\pm\infty)$ to be the slopes of the asymptotes at $\pm\infty$.
    This follows by extending \cref{l:duality} to $\pm\infty$, using (in the notation there) continuity of $u$ and compactness of $\Delta$.
\end{remark}
Moreover, a direct computation (using \cref{ic:diff} and the formula for the $L^q$-spectrum from \cite{zbl:1091.28005}) allows us to characterize precisely when the usual multifractal formalism holds.
See \cref{ss:lq-equal} for more details.
Note that the final condition equivalently states that the $L^q$-spectra of the column IFSs are all equal.
\begin{icorollary}\label{ic:coarse-formalism}
    Let $\mu_{\bm{p}}$ be a Gatzouras--Lalley measure with separation of principal projections.
    Then the following are equivalent.
    \begin{enumerate}[nl,r]
        \item\label{i:cf-1} $f_{\bm{p}}(\alpha) = \tau_{\bm{p}}^*(\alpha)$ for all $\alpha\in\R$.
        \item\label{i:cf-2} $\tau_{\bm{p}}(q) = \vartheta_{\bm{p}}(q)$ for all $q\in\R$.
        \item\label{i:cf-3} The function $\bm{v}\mapsto T_{\bm{p}}(\bm{v}, q)$ is constant.
    \end{enumerate}
\end{icorollary}
In general, however, $\vartheta_{\bm{p}}$ need not be differentiable.
In such cases, we need a way to bound $f_{\bm{p}}$ without using the fine $L^q$-spectrum and duality.
Unfortunately, we cannot do this in general, so we must restrict our attention to Gatzouras--Lalley measures with two columns.
\begin{itheorem}\label{it:two-col-var}
    Let $\mu_{\bm{p}}$ be a Gatzouras--Lalley measure with separation of principal projections and with two columns.
    Then for all $\alpha\in\R$,
    \begin{align*}
        f_{\bm{p}}(\alpha)
        &=\max_{\bm{w}\in\mathcal{P}}\left\{v(\bm{w}):u_{\bm{p}}(\bm{w}) = \alpha\right\}\\
        &=\max_{\bm{v}\in\eta(\mathcal{P})}
        \left\{
            \frac{H(\bm{v})}{\chi_1(\bm{v})}
            +
            T^*_{\bm{p}}\left(\bm{v}, \alpha-\frac{\CH{\bm{v}}{\eta(\bm{p})}}{\chi_1(\bm{v})}\right)
        \right\}.
    \end{align*}
\end{itheorem}
In the above result, the sets in the maximum are non-empty for a compact interval of $\alpha$.
Similarly to the convention for the multifractal, the maximum over the emptyset is $-\infty$.

Finally we can give examples exhibiting exceptional behaviour, already in the case that the IFS has two columns and three maps.
\begin{itheorem}\label{it:two-col-exc}
    There exists a Gatzouras--Lalley measure on $\R^2$ with separation of principal projections such that $f_{\bm{p}}$ has a jump discontinuity in the interior of its support.
    In particular, the multifractal formalism fails: there is a number $\alpha>0$ such that $f_{\bm{p}}(\alpha) < \vartheta_{\bm{p}}^*(\alpha)$.
\end{itheorem}
\begin{remark}
    In our examples, we must necessarily exploit non-uniqueness of the minimizing probability vectors in the formula in \cref{it:main} (first observed, in the case $q=0$, in \cite{zbl:1228.37025}).
    It follows from general facts about optimization (see for instance \cref{l:duality}) that if the minimum in the variational formula for $\vartheta_{\bm{p}}(q)$ is attained at a unique value $\bm{w}\in\mathcal{P}$, then $\vartheta_{\bm{p}}'(q) = u_{\bm{p}}(\bm{w})$ and therefore the multifractal formalism holds at $q$.
    Of course, it could happen that there are multiple minimizers but $\vartheta_{\bm{p}}'$ still exists: this is the case precisely when $u_{\bm{p}}$ is constant across all of the minimizers.
\end{remark}

\section{Background on Gatzouras--Lalley measures}\label{s:not}
In this section, we introduce standard background and context for Gatzouras--Lalley measures.
\subsection{Gatzouras--Lalley measures and separation conditions}\label{ss:GL}
Fix an index set $\mathcal{I}$ with $\#\mathcal{I}\geq 2$, and for $j=1,2$ fix contraction ratios $(\ctr_{i,j})_{i\in\mathcal{I}}$ in $(0,1)$ and translations $(d_{i,j})_{i\in\mathcal{I}}\in\R^{\mathcal{I}}$.
We then call the IFS $\Phi=\{T_i\}_{i\in\mathcal{I}}$ \defn{diagonal} if
\begin{equation*}
    T_i(x_1,x_2)=(\ctr_{i,1} x_1+d_{i,1},\ctr_{i,2} x_2+d_{i,2})\text{ for each }i\in\mathcal{I}.
\end{equation*}
Without loss of generality, we may assume that $T_i((0,1)^2) \subset (0,1)^2$.
We let $\Lambda$ denote the unique non-empty compact invariant set.

Let $\eta$ denote the projection onto the $1$\textsuperscript{st} coordinate axis, i.e.\ $\eta(x_1,x_2)=x_1$.
We denote by $\eta\Phi=\{S_{\underline{i}}\}_{\underline{i}\in\eta(\mathcal{I})}$ the \defn{projected system}, where $\eta\circ T_i=S_{i}\circ\eta$.
Of course, $S_{i}(x)=\ctr_{i,1} x+d_{i,1}$ are similarity maps.
In this way, the map $\eta$ induces an equivalence relation on the index set $\mathcal{I}$ where $i\sim j$ if $S_{i}=S_{j}$.
One might think of elements $\underline{i}\in\eta(\mathcal{I})$ as referring to columns of the corresponding IFS.

Let $\mathcal{I}^*=\bigcup_{n=0}^\infty\mathcal{I}^n$, and for $\mtt{i}=(i_1,\ldots,i_n)\in\mathcal{I}^*$, write
\begin{align*}
    T_{\mtt{i}} &= T_{i_1}\circ\cdots\circ T_{i_n}\\
    S_{\mtt{i}} &= S_{i_1}\circ\cdots\circ S_{i_n}
\end{align*}
and
\begin{align*}
    p_{\mtt{i}} &= p_{i_1}\cdots p_{i_n}\\
    \ctr_{\mtt{i},j} &= \ctr_{i_1,j}\cdots \ctr_{i_n,j}
\end{align*}
For $n\in\N$ and $\gamma\in\mathcal{I}^{\N}$, we write $\gamma\npre{n}$ to denote the unique prefix of $\gamma$ in $\mathcal{I}^{n}$.

Throughout, we will denote the unique non-empty compact invariant set associated with the IFS $\Phi$ by $\Lambda$.
Then the projected IFS $\eta\Phi$ has invariant set $\eta(\Lambda)$.
\begin{definition}
    We say that the diagonal IFS $\{T_i\}_{i\in\mathcal{I}}$ is \emph{Gatzouras--Lalley} if the following additional assumptions hold:
    \begin{enumerate}[nl,r]
        \item $\ctr_{i,1} > \ctr_{i,2}$ for all $i\in\mathcal{I}$.
        \item $T_i((0,1)^2) \cap T_j((0,1)^2) = \varnothing$ for all $i \neq j$ in $\mathcal{I}$.
        \item $S_{\underline{i}}((0,1)) \cap S_{\underline{j}}((0,1)) = \varnothing$ for all $\underline{i} \neq \underline{j}$ in $\eta(\mathcal{I})$.
    \end{enumerate}
    We call its invariant set $\Lambda$ a \emph{Gatzouras--Lalley carpet} and each associated self-affine measure $\mu_{\bm{p}}$ a \emph{Gatzouras--Lalley measure}.
\end{definition}
Whenever we refer to the reference Gatzouras--Lalley measure $\mu_{\bm{p}}$, by deleting extra maps if necessary, we may assume that $\supp \mu_{\bm{p}} = \Lambda$.
Of course, other probability vectors in $\mathcal{P}$ need not be fully supported.

Unfortunately, we will require an additional technical separation assumption.
Removing this technical assumption seems difficult (compare, for instance, \cite{zbl:0845.28007} and \cite{zbl:1206.28012}).
\begin{definition}\label{d:spp}
    We say that the IFS $\{T_i\}_{i\in\mathcal{I}}$ has \emph{(strong) separation of principal projections} if $T_i(\Lambda) \cap T_j(\Lambda) = \varnothing$ for all $i \neq j$ in $\mathcal{I}$ and $S_{\underline{i}}(\eta(\Lambda)) \cap S_{\underline{j}}(\eta(\Lambda)) = \varnothing$ for all $\underline{i} \neq \underline{j}$ in $\eta(\mathcal{I})$.
    We say that a Gatzouras--Lalley measure has separation of principal projections if its defining IFS does.
\end{definition}
\begin{remark}
    In the literature, sometimes the terminology ``projective separation'' is used to enforce that the projected maps are all distinct; for instance, this is the case in \cite{zbmath:8179359,zbl:0867.28006}.
    Here, we only require that the projected maps are either separated or overlap exactly.
\end{remark}

\subsection{Invariant measures and entropy}\label{ss:ent}
More generally, we consider invariant measures supported on the set $\Lambda$.
Let $\pi\colon\mathcal{I}^{\N}\to \Lambda$ be defined by the rule
\begin{equation*}
    \bigl\{\pi\bigl((i_n)_{n=1}^\infty\bigr)\bigr\}=\lim_{n\to\infty}T_{i_1}\circ\cdots\circ T_{i_n}(\Lambda).
\end{equation*}
Let $\mathcal{P}(\mathcal{I})$ denote the collection of probability vectors on $\mathcal{I}$, i.e.
\begin{equation*}
    \mathcal{P}=\mathcal{P}(\mathcal{I})\coloneqq\{(p_i)_{i\in\mathcal{I}}:p_i\geq 0\text{ for all }i,\sum_{i\in\mathcal{I}}p_i=1\}.
\end{equation*}
We also denote the corresponding simplex on the index set $\eta(\mathcal{I})$ by $\eta(\mathcal{P})$, where $\eta$ is the natural map $\eta\colon\mathcal{P}\to\eta(\mathcal{P})$ defined by the rule
\begin{equation*}
    \eta(\bm{w})=\left(\sum_{\substack{i\in\eta^{-1}(\underline{\ell})}}\bm{w}_i\right)_{\underline{\ell}\in\eta(\mathcal{I})}.
\end{equation*}

Given $\bm{p}\in\mathcal{P}$, considering $\bm{p}$ as a probability measure on $\mathcal{I}$, we let $\nu_{\bm{p}}$ denote the infinite product measure $\bm{p}^{\N}$ supported on $\mathcal{I}^{\N}$.
We let $\mu_{\bm{p}}=\pi_*\nu_{\bm{p}}$ denote the corresponding Gatzouras--Lalley measure on $\Lambda$, where $\pi_*$ denotes the pushforward map.
We also set
\begin{equation*}
    \Omega_{\bm{p}}=\Bigl\{(i_n)_{n=1}^\infty\in\mathcal{I}^{\N}:\lim_{n\to\infty}\frac{\#\{\ell:i_\ell=j\text{ for }1\leq \ell\leq n\}}{n}=p_j\text{ for }j\in\mathcal{I}\Bigr\},
\end{equation*}
in other words the collection of sequences in $\mathcal{I}^{\N}$ where the digit frequencies exist and are given by the probability vector $\bm{p}$.

Next, for any $\bm{w},\bm{q}\in\mathcal{P}$, we define the \defn{cross entropy}
\begin{equation*}
    \CH{\bm{w}}{\bm{q}}=\sum_{i\in\supp\bm{w}}w_i\log(1/q_i),
\end{equation*}
and \defn{divergence}
\begin{equation*}
    \DKL{\bm{w}}{\bm{q}}=\sum_{i\in\supp\bm{w}}w_i\log(w_i/q_i).
\end{equation*}
Here, $\log$ takes values in the extended real line; these quantities are finite precisely when $\supp\bm{w}\subseteq\supp\bm{q}$.
We let $H(\bm{w})=\CH{\bm{w}}{\bm{w}}$ denote the \defn{entropy} of $\bm{w}$.
We recall in general that
\begin{equation*}
    0\leq\DKL{\bm{w}}{\bm{q}}=\CH{\bm{w}}{\bm{q}}-H(\bm{w}).
\end{equation*}
Finally, for $j=1,2$, we define the \defn{Lyapunov exponents}
\begin{equation*}
    \chi_j(\bm{w})=\sum_{i\in\mathcal{I}}w_i\log(1/\ctr_{i,j}),
\end{equation*}
Since $\chi_1(\bm{w})$ depends only on $\eta(\bm{w})$ we will also write $\chi_1(\eta(\bm{w}))$ when relevant.
\subsection{Approximate squares}
A key to understanding the geometry of the self-affine carpet $\Lambda$ is the notion of an \emph{approximate square}.
First, given $\mtt{i}\in\mathcal{I}^n$ and $\underline{\mtt{j}}\in\eta(\mathcal{I})^m$, the corresponding \emph{pseudo-cylinder} is the set
\begin{equation*}
    P(\mtt{i}, \underline{\mtt{j}}) = \{\gamma = (i_k)_{k=1}^\infty\in\mathcal{I}^{\N}: (i_1,\ldots,i_n) = \mtt{i}\text{ and }\eta(i_{n+1},\ldots, i_{n+m}) = \underline{\mtt{j}}\}.
\end{equation*}
Note that the map $(\mtt{i}, \underline{\mtt{j}})\mapsto P(\mtt{i}, \underline{\mtt{j}})$ is injective.
An equivalent way to think of the pseudo-cylinder is as the union of the cylinders contained in $[\mtt{i}]$ which lie in the column $\eta(\mtt{i})\underline{\mtt{j}}$:
\begin{equation*}
    P(\mtt{i}, \underline{\mtt{j}}) = \bigcup_{\mtt{k}\in\eta^{-1}(\underline{\mtt{j}})}[\mtt{i}\mtt{k}].
\end{equation*}
Here, $[\mtt{i}]$ is the set of those words $\gamma\in\mathcal{I}^{\N}$ which contain $\mtt{i}$ as a prefix.

Now, fix an infinite word $\gamma\in\mathcal{I}^{\N}$ and some $k\in\N$, and let $\mtt{i} = \gamma\npre{k}$.
The cylinder $[\mtt{i}]$ corresponds to a rectangle which is wider than it is tall.
Let $L_k(\gamma)\geq k$ be the minimal integer so that
\begin{equation*}
    \ctr_{\gamma\npre{L_k(\gamma)}, 1}<\ctr_{\mtt{i}, 2}
\end{equation*}
In other words, $L_k(\gamma)$ is chosen so that the level $L_k(\gamma)$ rectangle has approximately the same width as the height of the level $k$ rectangle.
Writing $\gamma\npre{L_k(\gamma)} = \mtt{i}\mtt{j}$, we then define the \emph{level $k$ approximate square}
\begin{equation*}
    Q_k(\gamma) = P(\mtt{i}, \eta(\mtt{j})).
\end{equation*}
Observe that
\begin{equation*}
    \diam(\pi(Q_k(\gamma)))\approx\ctr_{\gamma_1,2}\cdots\ctr_{\gamma_k,2}\approx\ctr_{\gamma_1,1}\cdots\ctr_{\gamma_{L_k(\gamma)},1},
\end{equation*}
where $\approx$ means up to multiplication by fixed constants depending only on the IFS.

It turns out that we can obtain a convenient formula for the $\mu_{\bm{w}}$-measure of an approximate square.
We first introduce some notation for the truncated types at level $k$.
Given a finite word $\mtt{i} = (i_1,\ldots,i_n)\in\mathcal{I}^n$, $\bm{\xi}(\mtt{i})\in\mathcal{P}$ is the empirical probability vector
\begin{equation*}
    \bm{\xi}(\mtt{i})=\left(\frac{\#\{j:i_j=i\}}{n}\right)_{i\in\mathcal{I}}.
\end{equation*}
We then set for $k\in\N$
\begin{equation*}
    \bm{\lambda}_k(\gamma) = \bm{\xi}((\gamma_1,\ldots,\gamma_k)).
\end{equation*}
In other words, $\bm{\lambda}_k(\gamma)$ is the empirical digit frequency in the first $k$ digits of $\gamma$.
We obtain the following lemma.
\begin{lemma}\label{l:square-measure}
    There is a bounded function $c(k,\gamma,\bm{w})$ such that for any $k\in\N$, $\gamma\in\mathcal{I}^{\N}$ and $\bm{w}\in\mathcal{P}$ with $\nu_{\bm{w}}(Q_k(\gamma))>0$, writing $\bm{\lambda}_n=\bm{\lambda}_n(\gamma)$ for $n\in\N$,
    \begin{equation*}
        \frac{\log\nu_{\bm{w}}Q_k(\gamma))}{c(k,\gamma,\bm{w})+\log\diam(\pi(Q_k(\gamma)))}=
        \frac{\CH{\eta(\bm{\lambda}_{L_k})}{\eta(\bm{w})}}{\chi_1(\eta(\bm{\lambda}_{L_k}))}
        +\frac{\CH{\bm{\lambda}_k}{\bm{w}}-\CH{\eta(\bm{\lambda}_k)}{\eta(\bm{w})}}{\chi_2(\bm{\lambda}_k)}
    \end{equation*}
\end{lemma}
\begin{proof}
    Throughout, write $L_k(\gamma)=L_k$ and $\bm{\lambda}_k(\gamma)=\bm{\lambda}_k$, and set $\bm{q}=\eta(\bm{w})$.
    We recall for each $k\in\N$ that
    \begin{equation*}
        \nu_{\bm{w}}(Q_k)=\bm{w}_{\gamma_1}\cdots \bm{w}_{\gamma_k}\cdot \bm{q}_{\eta(\gamma_{k+1})}\cdots \bm{q}_{\eta(\gamma_{L_k})}=\frac{\prod_{i\in\mathcal{I}}\bm{w}_i^{k(\bm{\lambda}_k)_i}\cdot\prod_{\ell\in\eta(\mathcal{I})}\bm{q}_\ell^{L_k(\eta(\bm{\lambda}_{L_k}))_\ell}}{\prod_{\ell\in\eta(\mathcal{I})}\bm{q}_\ell^{k(\eta(\bm{\lambda}_k))_\ell}}
    \end{equation*}
    and
    \begin{align*}
        \diam(Q_k)&\approx \beta_{\gamma_1,2}\cdots\beta_{\gamma_k,2}=\prod_{i\in\mathcal{I}}\beta_{i,2}^{k(\bm{\lambda}_k)_i}\\
                  &\approx \beta_{\gamma_1,1}\cdots\beta_{\gamma_{L_k},1}=\prod_{i\in\mathcal{I}}\beta_{i,1}^{L_k(\bm{\lambda}_{L_k})_i}
    \end{align*}
    with constants not depending on $k$.
    Taking logarithms and dividing yields the desired result.
\end{proof}
We conclude this section with a fundamental property of separation of principal projections, which ensures that measures of approximate squares $Q_k(\gamma)$ and balls $B\bigl(\pi(\gamma),\diam \pi(Q_k(\gamma))\bigr)$ are the same up to a fixed constant factor.
\begin{lemma}\label{l:spsc-reduction}
    Suppose that the Gatzouras--Lalley IFS has separation of principal projections.
    Write $x=\pi(\gamma)$ and $r_k=\ctr_{\gamma\npre{k},2}$.
    There are constants $c,C>0$ (depending only on the IFS) such that
    \begin{equation*}
        B(x,cr_k)\cap \Lambda\subset\pi(Q_k(\gamma))\subset B(x,Cr_k).
    \end{equation*}
\end{lemma}
\begin{proof}
    Let $\delta>0$ be sufficiently small so that the sets $T_i(\Lambda)$ and $\eta(T_i(\Lambda))$ are either identical or $\delta$-separated.
    If $\gamma'\notin Q_k(\gamma)$ then either $\gamma\npre{k} \neq \gamma'\npre{k}$ or $\eta(\gamma\npre{L_k(\gamma)}) \neq \eta(\gamma'\npre{L_k(\gamma)})$.
    In either case, it follows that $|\pi(\gamma') - x| \geq \delta r_k$.
    This proves the first inequality; the second is immediate from the definition of an approximate square.
\end{proof}

\section{Background on multifractal analysis and duality}
\subsection{The fine \texorpdfstring{$L^q$}{Lq}-spectrum}\label{ss:lq}
We first recall the usual generalized multifractal Hausdorff measures from \cite{zbl:0841.28012}.
Fix a Borel probability measure $\mu$.
Then for $E\subset\supp\mu$, define
\begin{equation*}
    \overline{\mathcal{H}}_{\mu,\delta}^{q,t}(E)=\inf\left\{\sum_{i=1}^\infty\mu\bigl(B(x_i,r_i)\bigr)^q r_i^t:x_i\in E,E\subset\bigcup_{i=1}^\infty B(x_i,r_i),r_i\leq\delta\right\}.
\end{equation*}
Then write
\begin{equation*}
    \mathcal{H}_\mu^{q,t}(E)=\sup_{F\subset E}\sup_{\delta>0}\overline{\mathcal{H}}_{\mu,\delta}^{q,t}(F)
\end{equation*}
and set
\begin{equation*}
    b_{\mu}(E,q)=\inf\{t:\mathcal{H}_\mu^{q,t}(E) < \infty\}.
\end{equation*}
Note that when $q=0$, $b_\mu(E,0)=\dimH E$ is independent of $\mu$.

Finally, write $\vartheta_\mu(q)=-b_\mu(\supp\mu,q)$.
We will call the function $\vartheta_\mu(q)$ the \emph{fine $L^q$-spectrum}, in contrast to the \emph{$L^q$-spectrum} $\tau_\mu(q)$ defined earlier.
Of course, $\vartheta_\mu(0)=-\dimH\supp\mu$.
Unfortunately, $\vartheta_\mu$ need not be a concave function of $q$ in general; see the construction in \cite[Section~3]{zbl:0841.28012}.

We also recall from \cite{zbl:0841.28012} that $\vartheta_\mu$ gives a general upper bound for the multifractal spectrum.
Let
\begin{equation*}
    f_\mu(\alpha)\coloneqq \dimH\{x\in\supp\mu:\dim_{\loc}(\mu, x) = \alpha\}.
\end{equation*}
Here, we write $\dimH\varnothing = -\infty$.

We will require exactly one general fact about the fine $L^q$-spectrum.
This is a special case of \cite[Theorem~2.17]{zbl:0841.28012} (for the endpoints, use \cite[Proposition~2.5(iii) and Lemma~4.4]{zbl:0841.28012}).
\begin{proposition}\label{p:spec-upper}
    Let $\mu$ be a Borel probability measure on $\R^d$.
    Then for all $\alpha\in\R$,
    \begin{equation*}
        f_\mu(\alpha) \leq \vartheta_\mu^*(\alpha).
    \end{equation*}
\end{proposition}

\subsection{Optimisation geometry}\label{ss:optimisation}
In this section, we recall the setup and collect the relevant results from \cite[§3.1]{arxiv:2312.08974}.

We begin by recalling some general definitions from convex analysis.
For a general function $g\colon\R\to\R\cup\{-\infty\}$, the \emph{concave conjugate} is given by
\begin{equation*}
    g^*(\alpha) = \inf_{t\in\R}(t\alpha-g(t)).
\end{equation*}
Note that $g^*$ is always concave, and $g^{**}$ is the concave hull of $g$.
In particular, we always have the inequality
\begin{equation}\label{e:subdiff-rel}
    g^*(\alpha)+g(t)\leq\alpha t,
\end{equation}
and moreover the \emph{subdifferential} $\partial g(t)$ is precisely the set of $\alpha$ for which equality holds in \cref{e:subdiff-rel}.

Suppose moreover that $g$ is a concave function.
We then let $\partial^-g(t)$ (resp.\ $\partial^+g(t)$) denote the left (resp.\ right) derivative of $g$ at $t$, which necessarily exist by concavity of $g$.
Equivalently, $\partial g(t)=[\partial^+g(t),\partial^-g(t)]$.
In particular, $g$ is differentiable at $t$ if and only if $\partial g(t)=\{\alpha\}$, in which case $g'(t)=\alpha$.

Now let $\Delta$ be a compact metric space and let $u,v\colon\Delta\to\R$ be continuous.
We are interested in the \emph{constrained maximum}
\begin{equation*}
    F(\alpha)=\max_{\bm{w}\in\Delta}\left\{v(\bm{w}):u(\bm{w})=\alpha\right\},
\end{equation*}
where we set $F(\alpha) = -\infty$ if there does not exist $\bm{w} \in \Delta$ with $u(\bm{w}) = \alpha$.
Associated with this constrained optimisation problem is the \emph{unconstrained dual}
\begin{equation*}
    T(t)=\min_{\bm{w}\in\Delta}\left\{t\cdot u(\bm{w})-v(\bm{w})\right\}.
\end{equation*}
Since $T$ is a minimum of affine functions, $T$ is necessarily concave.
For each $t\in\R$, we denote the set of minimising vectors for $T(t)$ by
\begin{equation*}
    M(t) \coloneqq \left\{\bm{w}\in\Delta:t\cdot u(\bm{w})-v(\bm{w})=T(t)\right\}.
\end{equation*}
The following facts will be useful to us; the proofs are short and elementary and can be found in \cite[§3]{arxiv:2312.08974}.
\begin{lemma}[\cite{arxiv:2312.08974}]\label{l:duality}
    The following hold.
    \begin{enumerate}[nl,r]
        \item For all $\alpha\in\R$, $F(\alpha) \leq T^*(\alpha)$.
        \item If $\alpha = T'(q)$, then $F(\alpha) = T^*(\alpha)$.
        \item\label{im:subdiff-char} We have
            \begin{equation*}
                \min_{\bm{w}\in M(t)}u(\bm{w})=\partial^+T(t)\qquad\text{and}\qquad\max_{\bm{w}\in M(t)}u(\bm{w})=\partial^-T(t).
            \end{equation*}
        \item\label{im:conn} Suppose $t\in\R$ and $M(t)$ is connected.
            Then $u(M(t))=\partial T(t)$ and $F(\alpha)=T^*(\alpha)$ for all $\alpha\in\partial T(t)$.
        \item\label{im:sing} If $M(t)$ is a singleton, then the derivative $T'(t)$ exists.
            Moreover, if $M(t) = \{\bm{z}(t)\}$ is a singleton for all $t\in U$ where $U\subset\R$ is open, then $\bm{z}\colon U\to \Delta$ is continuous.
    \end{enumerate}
\end{lemma}

\subsection{Quasi-concavity}\label{ss:qc}
We now introduce a few geometric concepts which will be useful when working with optimization problems.
For a more detailed introduction, see for instance \cite[§3.4]{zbl:1058.90049}.
\begin{definition}
    Let $\Delta$ be a convex space and let $f\colon\Delta\to\R$.
    We say that $f$ is \emph{quasi-concave} if for all $t\in\R$ the upper level set
    \begin{equation*}
        \{x\in\Delta:f(x)\geq t\}
    \end{equation*}
    is convex.
    Moreover, we say that $f$ is \emph{strictly quasi-concave} if for all $x \neq y\in\Delta$ with $f(x) = f(y)=t$ and $\lambda\in(0,1)$, $f(\lambda x+ (1-\lambda) y) > t$.
    Similarly, $f$ is \emph{(strictly) quasi-convex} if $-f$ is (strictly) quasi-concave.
\end{definition}
An important concept in the dimension theory of self-affine sets is the Ledrappier--Young formula \cite{zbl:1230.37031,zbl:1371.37012,zbl:0605.58028}.
In the most basic cases, it can be expressed as a ratio $H(\bm{w})/\chi(\bm{w})$; in our case, recall the formulas from \cref{e:ly-formulas}.
Such ratios are always quasi-concave.
\begin{lemma}\label{l:qc-const}
    Let $\Delta$ be a convex space.
    Suppose $f\colon\Delta\to\R$ is concave (resp.\ strictly concave) and $g\colon\Delta\to(0,\infty)$ is affine.
    Then $f/g$ is quasi-concave (resp.\ strictly quasi-concave).
    Similarly, if $f$ is convex (resp.\ strictly convex), then $f/g$ is quasi-convex (resp.\ strictly quasi-convex).
\end{lemma}
\begin{proof}
    We handle the concave case; the quasi-convex case follows analogously.
    Suppose $f$ is concave and $g$ is affine.
    Then for each $t\in\R$, using positivity of $g$,
    \begin{equation*}
        \{x\in\Delta:f(x)/g(x)\geq t\}=\{x\in\Delta:f(x)-t g(x)\geq 0\}.
    \end{equation*}
    This is a convex set since $f(x)-t g(x)$ is concave as a function of $x$.
    Moreover, if $f$ is strictly concave, then for fixed $t$, $f(x) - t g(x)$ is strictly concave so for all $x\neq y$ and $\lambda\in(0,1)$ with $f(x)/g(x) = f(y)/g(y) = t$,
    \begin{equation*}
        f\bigl(\lambda x + (1-\lambda) y\bigr) - t g\bigl(\lambda x + (1-\lambda) y\bigr) > 0
    \end{equation*}
    so $f/g$ is strictly quasi-concave.
\end{proof}
\begin{remark}
    Unlike sums of concave functions, sums of quasi-concave functions need not be quasi-concave.
    This simple fact is at the heart of the exceptional phenomena which begin to occur in the dimension theory of self-affine sets and measures in dimensions 2 and above.
\end{remark}
Applying \cref{l:duality}, we immediately obtain the following application.
\begin{lemma}\label{l:qc-min}
    Let $\Delta$ be a compact convex space and let $u,v\colon\Delta\to\R$ be continuous.
    Consider the dual functions
    \begin{align*}
        F(\alpha)&=\max_{\bm{w}\in\Delta}\left\{v(\bm{w}):u(\bm{w})=\alpha\right\},\\
        T(t)&=\min_{\bm{w}\in\Delta}\left\{t\cdot u(\bm{w})-v(\bm{w})\right\}.
    \end{align*}
    Suppose $q\in \R$ and $\bm{w}\mapsto q\cdot u(\bm{w})-v(\bm{w})$ is quasi-convex.
    Then $F(\alpha) = T^*(\alpha)$ for all $\alpha\in\partial T(q)$.
    Moreover, if $\bm{w}\mapsto q\cdot u(\bm{w})-v(\bm{w})$ is strictly quasi-convex, then $T'(q)$ exists.
\end{lemma}

\section{Fine \texorpdfstring{$L^q$}{Lq}-spectra of Gatzouras--Lalley measures}

\subsection{Pointwise Billingsley's lemma}
In this section, we state the key technical density results which we will require in the later sections.

We first recall an equivalent version of \cite[Definition~3.12]{arxiv:2312.08974}.
\begin{definition}\label{d:unif}
    Let $\Delta$ be a family of Borel probability measures on $\R^d$.
    We say that $\Delta$ has \defn{uniform densities} if there is a finite Borel measure $\nu$ on $\R^d$ such that for every $\varepsilon>0$, there are constants $C>0$ and $\eta\in(0,1)$ so that for all $\mu\in\Delta$, $r\in(0,\eta)$, and $x\in\R^d$,
    \begin{equation*}
        r^\varepsilon\mu\bigl(B(x,r)\bigr)\leq\nu\bigl(B(x,Cr)\bigr).
    \end{equation*}
\end{definition}
Next, we recall \cite[Proposition~3.13]{arxiv:2312.08974}; there, the following result is only stated for self-similar IFSs but the proof works verbatim for any strictly contracting Lipschitz IFS.
\begin{lemma}[\cite{arxiv:2312.08974}]\label{l:meas-quasi-compact}
    Let $d\in\N$ and $\{T_i\}_{i\in\mathcal{I}}$ be a finite set of affine contractions from $\R^d$ to $\R^d$.
    Then the family of measures $\Delta\coloneqq\{\mu_{\bm{p}}:\bm{p}\in\mathcal{P}\}$ has uniform densities.
\end{lemma}
Finally, we need a multifractal generalization of \cite[Proposition~3.14]{arxiv:2312.08974}.
The proof is similar, but there are a few technical details which must be handled carefully.
Recall the definition $b_\lambda(E,q)$ from \cref{ss:lq}; for a first reading simply substitute $-\vartheta_\mu(q)$.
\begin{proposition}\label{p:pointwise-var}
    Let $\lambda$ be a compactly supported Borel probability measure in $\R^d$, let $E\subset\supp\lambda$, and let $\Delta$ be a family of measures with uniform densities.
    Let $q\in\R$ and suppose $t$ is such that for all $x\in E$,
    \begin{equation}\label{e:lld}
        \liminf_{r\to 0}\inf_{\mu\in\Delta}\frac{\log \mu\bigl(B(x,r)\bigr) - q \log \lambda\bigl(B(x,r)\bigr)}{\log r} \leq t.
    \end{equation}
    Then $b_\lambda(E,q) \leq t$.
\end{proposition}
\begin{proof}
    Let $\nu$ be the measure realizing the uniform densities of $\Delta$.
    Fix $\varepsilon>0$.
    By uniform densities, there are constants $C\geq 1$ and $\eta_0\in(0,1)$ such that
    \begin{equation}\label{e:discretize}
        r^\varepsilon\mu\bigl(B(x,r)\bigr)\leq\nu\bigl(B(x,Cr)\bigr)
    \end{equation}
    for all $\mu\in\Delta$, $x\in\R^d$, and $r\in(0,\eta_0)$.
    Let $F\subset E$ and $\eta\in(0,\eta_0)$ be arbitrary.
    For each $x\in F$, by \cref{e:lld}, choose $\mu_x\in\Delta$ and $r_x\in(0,\eta)$ such that
    \begin{equation}\label{e:mux-choice}
        r_x^{t+\varepsilon}\lambda\bigl(B(x,r_x)\bigr)^q\leq\mu_x\bigl(B(x,r_x)\bigr).
    \end{equation}
    Combining \cref{e:discretize,e:mux-choice} gives
    \begin{equation*}
        r_x^{t+2\varepsilon}\lambda\bigl(B(x,r_x)\bigr)^q\leq\nu\bigl(B(x,Cr_x)\bigr).
    \end{equation*}

    Since $F\subset\supp\lambda$ is bounded, we may apply the Besicovitch covering theorem (see \cite[Theorem~2.7]{zbl:0819.28004}) to $\mathcal{B}=\{B(x,r_x):x\in F\}$ to obtain families $\mathcal{B}_i\subset\mathcal{B}$, $i=1,\ldots,c_d$, each consisting of pairwise disjoint balls, whose union covers $F$.
    For $n\geq 0$, write
    \begin{equation*}
        \mathcal{B}_{i,n}=\{B(x,r)\in\mathcal{B}_i:2^{-n-1}<r\leq 2^{-n}\}.
    \end{equation*}
    Since the balls in each $\mathcal{B}_{i,n}$ are disjoint and have comparable radii, the balls $\{B(x,Cr):B(x,r)\in\mathcal{B}_{i,n}\}$ overlap boundedly with constant $D$ depending only on $d$ and $C$.
    Thus
    \begin{align*}
        \overline{\mathcal{H}}_{\lambda,\eta}^{q,t+3\varepsilon}(F)
        &\leq\sum_{i=1}^{c_d}\sum_{n=0}^\infty\sum_{B(x,r)\in\mathcal{B}_{i,n}}r^{t+3\varepsilon}\lambda\bigl(B(x,r)\bigr)^q\\
        &\leq\sum_{i=1}^{c_d}\sum_{n=0}^\infty 2^{-n\varepsilon}\sum_{B(x,r)\in\mathcal{B}_{i,n}}\nu\bigl(B(x,Cr)\bigr)\\
        &\leq c_dD\nu(\R^d)\sum_{n=0}^\infty 2^{-n\varepsilon}<\infty.
    \end{align*}
    The bound is independent of $F$ and $\eta$, so $\mathcal{H}_\lambda^{q,t+3\varepsilon}(E)<\infty$.
    Hence $b_\lambda(E,q)\leq t+3\varepsilon$, and letting $\varepsilon\to 0$ completes the proof.
\end{proof}
\subsection{Column pressure}
Let us begin recalling the definition from the introduction.
Define $\Psi\colon\eta(\mathcal{P})\times\R\times\R\to\R$ by
\begin{equation*}
    \Psi(\bm{v},q,s)=-\sum_{\ell\in\eta(\mathcal{I})}\bm{v}_\ell\log\Bigl(\sum_{j\in\eta^{-1}(\ell)}p_j^q\ctr_{j,2}^{-s}\Bigr).
\end{equation*}
Observe that $\Psi$ is continuous and strictly decreasing in $s$.
Moreover,
\begin{equation*}
    \lim_{s\to -\infty}\Psi(\bm{v},q,s)=\infty
    \qquad\text{and}\qquad
    \lim_{s\to\infty}\Psi(\bm{v},q,s)=-\infty.
\end{equation*}
In particular, this implies the following.
\begin{lemma}\label{l:Tvq-choice}
    Let $\bm{p}\in\mathcal{P}$.
    There is a unique continuous function $T_{\bm{p}}\colon\eta(\mathcal{P})\times\R\to\R$ satisfying
    \begin{equation*}
        \Psi\bigl(\bm{v},q,T_{\bm{p}}(\bm{v},q)\bigr)=q\CH{\bm{v}}{\eta(\bm{p})}.
    \end{equation*}
    for $q\in\R$ and $\bm{v}\in\eta(\mathcal{P})$.
    Moreover, $T_{\bm{p}}(\bm{v},1)=0$.
\end{lemma}
Now let $\bm{g}\colon\eta(\mathcal{P})\times\R\to\mathcal{P}$ be the continuous function defined for $\bm{v}\in\eta(\mathcal{P})$ and $q\in\R$ by the rule
\begin{equation*}
    \bm{g}(\bm{v},q)=\left(\bm{v}_{\eta(i)}\frac{p_i^q\ctr_{i,2}^{-T_{\bm{p}}(\bm{v},q)}}{\sum_{j\in\eta^{-1}(\eta(i))}p_j^q\ctr_{j,2}^{-T_{\bm{p}}(\bm{v},q)}}\right)_{i\in\mathcal{I}}.
\end{equation*}
Of course, $\eta(\bm{g}(\bm{v},q))=\bm{v}$.
The point of the choice of the functions $T_{\bm{p}}$ and $\bm{g}$ is the following lemma.
\begin{lemma}\label{l:fibre-lift}
    Let $\bm{p}\in\mathcal{P}$ be strictly positive and $q\in\R$.
    For any $\bm{w}\in\mathcal{P}$ and $\bm{v}\in\eta(\mathcal{P})$ such that $\supp\eta(\bm{w})\subseteq\supp\bm{v}$ and $T_{\bm{p}}(\eta(\bm{w}),q)\geq T_{\bm{p}}(\bm{v},q)$,
    \begin{equation*}
        \frac{q\bigl(\CH{\bm{w}}{\bm{p}}-\CH{\eta(\bm{w})}{\eta(\bm{p})}\bigr)+\CH{\eta(\bm{w})}{\bm{v}}-\CH{\bm{w}}{\bm{g}(\bm{v},q)}}{\chi_2(\bm{w})}\geq T_{\bm{p}}(\bm{v},q)
    \end{equation*}
    with equality if and only if $T_{\bm{p}}(\eta(\bm{w}),q)=T_{\bm{p}}(\bm{v},q)$.
\end{lemma}
\begin{proof}
    We compute that
    \begin{align*}
        \frac{\CH{\eta(\bm{w})}{\bm{v}}-\CH{\bm{w}}{\bm{g}(\bm{v},q)}}{\chi_2(\bm{w})} &= \frac{\sum_{i\in\mathcal{I}}\bm{w}_i\log\left(\frac{p_i^q\ctr_{i,2}^{-T_{\bm{p}}(\bm{v},q)}}{\sum_{j\in\eta^{-1}(\eta(i))}p_j^q\ctr_{j,2}^{-T_{\bm{p}}(\bm{v},q)}}\right)}{\chi_2(\bm{w})}\\
                                                                                         &=T_{\bm{p}}(\bm{v},q)-\frac{q\CH{\bm{w}}{\bm{p}}}{\chi_2(\bm{w})}+\frac{\Psi\bigl(\eta(\bm{w}),q,T_{\bm{p}}(\bm{v},q)\bigr)}{\chi_2(\bm{w})}\\
                                                                                         &\geq T_{\bm{p}}(\bm{v},q)-q\frac{\CH{\bm{w}}{\bm{p}}-\CH{\eta(\bm{w})}{\eta(\bm{p})}}{\chi_2(\bm{w})}.
    \end{align*}
    In the last line, we used \cref{l:Tvq-choice} and the fact that $\Psi$ is strictly decreasing in the third argument.
    In particular, this also implies the equality statement.
\end{proof}
Next, for $\bm{w}\in\mathcal{P}$ and $\bm{v}\in\eta(\mathcal{P})$, write
\begin{equation*}
    A(\bm{w}, \bm{v}) = \frac{H(\bm{w})-H(\bm{v})}{\chi_2(\bm{w})},
    \qquad
    B_{\bm{p}}(\bm{w}, \bm{v}) = \frac{\CH{\bm{w}}{\bm{p}}-\CH{\bm{v}}{\eta(\bm{p})}}{\chi_2(\bm{w})}.
\end{equation*}
For fixed $\bm{v}$, the function $T_{\bm{p}}(\bm{v}, q)$ can also be obtained as an explicit optimization problem.
\begin{corollary}\label{c:col-dual}
    We have
    \begin{equation*}
        T_{\bm{p}}(\bm{v},q)
        =
        \min_{\bm{w}\in\eta^{-1}(\bm{v})}
        \left\{
            q B_{\bm{p}}(\bm{w},\bm{v})-A(\bm{w}, \bm{v})
        \right\}
    \end{equation*}
    and moreover the minimum is attained uniquely at $\bm{g}(\bm{v}, q)$.
\end{corollary}
\begin{proof}
    Let $\bm{w}\in\eta^{-1}(\bm{v})$.
    By \cref{l:fibre-lift},
    \begin{equation*}
        q B_{\bm{p}}(\bm{w},\bm{v}) - A(\bm{w},\bm{v}) = T_{\bm{p}}(\bm{v}, q) + \frac{\DKL{\bm{w}}{\bm{g}(\bm{v}, q)}}{\chi_2(\bm{w})}.
    \end{equation*}
    To complete the proof, recall that $\DKL{\bm{w}}{\bm{g}(\bm{v}, q)} \geq 0$ with equality if and only if $\bm{w} = \bm{g}(\bm{v}, q)$.
\end{proof}

\subsection{Fine \texorpdfstring{$L^q$}{Lq}-spectra via fibred optimization}
Throughout, we fix a $\bm{p}\in\mathcal{P}$.
In this section, we prove a lower bound on the fine $L^q$-spectrum of the measure $\mu_{\bm{p}}$.
In order to prove this lower bound, we will use \cref{p:pointwise-var}: for each $x\in \Lambda$, we will choose a probability vector depending on the coding of $x$ and apply \cref{e:lld}.
By \cref{l:meas-quasi-compact}, the family of measures $\{\mu_{\bm{w}}:\bm{w}\in\mathcal{P}\}$ has uniform densities, which will give a lower bound on the fine $L^q$-spectrum.

The majority of the work is to choose the probability vector $\bm{w}$.
Unlike in the self-similar case, there is no single choice of a probability vector which works simultaneously for all $x\in\Lambda$.
The heart of the difficulty lies in the fact that the $\mu_{\bm{p}}$-measure of an approximate square $P(\mtt{i},\underline{\mtt{j}})$ depends separately on the coding of $\mtt{i}$ and $\underline{\mtt{j}}$, and the corresponding frequency vectors may not agree.
However, since there are only two relevant vectors, we can choose a sequence of scales along which the column spectrum $T_{\bm{p}}(\bm{v}, q)$ is minimized, and it turns out by \cref{l:fibre-lift} that such a sequence of scales suffices.

Given an infinite word $\gamma\in\mathcal{I}^{\N}$ and a $q\in\R$, let
\begin{equation*}
    \Phi(\gamma,q)=\liminf_{m\to\infty}T_{\bm{p}}(\eta(\bm{\lambda}_m(\gamma)),q).
\end{equation*}
We now obtain our main bound on the fine $L^q$-spectrum.
\begin{theorem}\label{t:lg-fine-lq-bound}
    Let $\bm{p}\in\mathcal{P}$ be strictly positive and let $q\in\R$, $\gamma\in\mathcal{I}^{\N}$, and $\delta>0$ be arbitrary.
    Then there is a fully supported $\bm{v}\in\eta(\mathcal{P})$ so that
    \begin{equation*}
        \limsup_{k\to\infty}\frac{q\log\nu_{\bm{p}}(Q_k(\gamma))-\log\nu_{\bm{g}(\bm{v},q)}(Q_k(\gamma))}{\log \diam\pi(Q_k(\gamma))}\geq\frac{q\CH{\bm{v}}{\eta(\bm{p})}-H(\bm{v})}{\chi_1(\bm{v})}+T_{\bm{p}}(\bm{v},q)-\delta.
    \end{equation*}
\end{theorem}
\begin{proof}
    For notational clarity, we suppress dependence of our notation on $\gamma$.
    Let $(k_m)_{m=1}^\infty$, $\bm{w}\in\mathcal{P}$, and $\bm{v}\in\eta(\mathcal{P})$ be chosen so that:
    \begin{itemize}[nl]
        \item $\Phi(\gamma,q)=T_{\bm{p}}(\bm{v},q)$,
        \item $\bm{v}=\lim_{m\to\infty}\eta(\bm{\lambda}_{L_{k_m}})$, and
        \item $\bm{w}=\lim_{m\to\infty}\bm{\lambda}_{k_m}$.
    \end{itemize}
    Since $L_k=O(k)$ and $k\eta(\bm{\lambda}_k)\leq L_k\eta(\bm{\lambda}_{L_k})$ coordinatewise, we have $\supp\eta(\bm{w})\subseteq\supp\bm{v}$.
    For $\varepsilon\in(0,1)$, set
    \begin{equation*}
        \bm{v}_\varepsilon=(1-\varepsilon)\bm{v}+\varepsilon\eta(\bm{p}),
    \end{equation*}
    so that $\bm{g}(\bm{v}_\varepsilon,q)$ is fully supported.
    Since $T_{\bm{p}}(\eta(\bm{w}),q)\geq T_{\bm{p}}(\bm{v},q)$, \cref{l:square-measure} and \cref{l:fibre-lift} give
    \begin{align*}
        \lim_{\varepsilon\to 0}\lim_{m\to\infty}
        &\frac{q\log\nu_{\bm{p}}(Q_{k_m})-\log\nu_{\bm{g}(\bm{v}_\varepsilon,q)}(Q_{k_m})}{\log\diam\pi(Q_{k_m})}\\
        ={}&\frac{q\CH{\bm{v}}{\eta(\bm{p})}-H(\bm{v})}{\chi_1(\bm{v})}\\
        &+\frac{q\bigl(\CH{\bm{w}}{\bm{p}}-\CH{\eta(\bm{w})}{\eta(\bm{p})}\bigr)+\CH{\eta(\bm{w})}{\bm{v}}-\CH{\bm{w}}{\bm{g}(\bm{v},q)}}{\chi_2(\bm{w})}\\
        \geq{}&\frac{q\CH{\bm{v}}{\eta(\bm{p})}-H(\bm{v})}{\chi_1(\bm{v})}+T_{\bm{p}}(\bm{v},q).
    \end{align*}
    By continuity of the final expression in $\bm{v}$, taking $\varepsilon>0$ sufficiently small gives the required inequality with $\bm{v}_\varepsilon$ in place of $\bm{v}$.
\end{proof}
The lower bound for the fine $L^q$ spectrum in \cref{it:main} now follows.
\begin{corollary}\label{c:lq-lower}
    Let $\bm{p}\in\mathcal{P}$ be fixed and $q\in\R$.
    Then
    \begin{align*}
        \vartheta_{\bm{p}}(q)&\geq\min_{\bm{w}\in\mathcal{P}}\left\{\frac{q \CH{\bm{w}}{\bm{p}}-H(\bm{w})}{\chi_2(\bm{w})}+\left(1-\frac{\chi_1(\bm{w})}{\chi_2(\bm{w})}\right)\frac{q\CH{\eta(\bm{w})}{\eta(\bm{p})}-H(\eta(\bm{w}))}{\chi_1(\eta(\bm{w}))}\right\}\\
                             &=\min_{\bm{v}\in\eta(\mathcal{P})}\left\{\frac{q\CH{\bm{v}}{\eta(\bm{p})}-H(\bm{v})}{\chi_1(\bm{v})}+T_{\bm{p}}(\bm{v},q)\right\}.
    \end{align*}
\end{corollary}
\begin{proof}
    That the two minima are equal follows from \cref{c:col-dual}; denote the common value by $s$.
    Fix $\delta>0$ and, for each $x=\pi(\gamma)$, choose a fully supported $\bm{v}$ as in \cref{t:lg-fine-lq-bound}.
    By \cref{l:spsc-reduction},
    \begin{equation*}
        \liminf_{r\to0}\frac{\log\mu_{\bm{g}(\bm{v},q)}\bigl(B(x,r)\bigr)-q\log\mu_{\bm{p}}\bigl(B(x,r)\bigr)}{\log r}\leq -s+\delta.
    \end{equation*}
    Apply \cref{p:pointwise-var} with $\lambda=\mu_{\bm{p}}$, $E=\supp\mu_{\bm{p}}$, $t=-s+\delta$, and the family $\Delta=\{\mu_{\bm{z}}:\bm{z}\in\mathcal{P}\}$, which has uniform densities by \cref{l:meas-quasi-compact}.
    Thus $\vartheta_{\bm{p}}(q)\geq s-\delta$.
    Since $\delta>0$ was arbitrary, $\vartheta_{\bm{p}}(q) \geq s$ as claimed.
\end{proof}

\subsection{A lower bound for the multifractal spectrum}
We now turn our attention to the multifractal spectrum.
As usual, we can obtain a simple lower bound for the multifractal spectrum as a constrained optimization problem.
Recall the definitions of $v$ and $u_{\bm{p}}$ from \cref{e:ly-formulas}.
The following lower bound on the multifractal spectrum is standard.
\begin{proposition}\label{p:mf-lower}
    Let $\mu_{\bm{p}}$ be a Gatzouras--Lalley measure with separation of principal projections, and let $\alpha \geq 0$.
    Then
    \begin{equation*}
        f_{\bm{p}}(\alpha) \geq F_{\bm{p}}(\alpha) \coloneqq \max_{\bm{w}\in\mathcal{P}}\left\{v(\bm{w}):u_{\bm{p}}(\bm{w}) = \alpha\right\}.
    \end{equation*}
\end{proposition}
\begin{proof}
    Fix $\bm{w}\in\mathcal{P}$ with $u_{\bm{p}}(\bm{w})=\alpha$.
    For $\nu_{\bm{w}}$-a.e.\ $\gamma$, the empirical frequencies $\bm{\lambda}_k(\gamma)$ converge to $\bm{w}$.
    Thus \cref{l:square-measure,l:spsc-reduction} give
    \begin{equation*}
        \dim_{\loc}(\mu_{\bm{p}},\pi(\gamma))=u_{\bm{p}}(\bm{w})=\alpha,
        \qquad
        \dim_{\loc}(\mu_{\bm{w}},\pi(\gamma))=v(\bm{w}).
    \end{equation*}
    Taking the maximum over $\bm{w}$ proves the claim.
\end{proof}
Crucially, this simple lower bound is already sufficient to give our exact formula for $\vartheta_{\bm{p}}$.
Essentially, this is because of duality.
Recall that we proved in \cref{c:lq-lower} that
\begin{equation*}
    \vartheta_{\bm{p}}(q) \geq V_{\bm{p}}(q)\coloneqq \min_{\bm{w}\in\mathcal{P}}\left\{q u_{\bm{p}}(\bm{w}) - v(\bm{w})\right\}.
\end{equation*}
The optimization problems defining $V_{\bm{p}}$ and $F_{\bm{p}}$ are dual in the sense of \cref{ss:optimisation}.
In particular, we know in general that $V_{\bm{p}}$ is a concave function of $q$ with $F_{\bm{p}}^*(q) = V_{\bm{p}}(q)$.

However, in general, $F_{\bm{p}}(\alpha)$ need not be concave, so it can happen that $F(\alpha) < V_{\bm{p}}^*(\alpha)$; see \cref{c:excp}.
Regardless, using the general inequalities along with the order-reversing property of concave conjugation, we can obtain our variational formula for $\vartheta_{\bm{p}}$ and complete the proof of \cref{it:main}.
\begin{theorem}
    Let $\bm{p}\in\mathcal{P}$ be fixed and $q\in\R$.
    Then
    \begin{equation*}
        \vartheta_{\bm{p}}(q) = V_{\bm{p}}(q) =\min_{\bm{v}\in\eta(\mathcal{P})}\Bigl(\frac{q\CH{\bm{v}}{\eta(\bm{p})}-H(\bm{v})}{\chi_1(\bm{v})}+T_{\bm{p}}(\bm{v},q)\Bigr).
    \end{equation*}
\end{theorem}
\begin{proof}
    The second equality and the inequality $\vartheta_{\bm{p}}(q) \geq V_{\bm{p}}(q)$ are precisely \cref{c:lq-lower}.
    Also, recall \cref{p:mf-lower} and \cref{p:spec-upper}:
    \begin{equation}\label{e:gen}
        F_{\bm{p}}(\alpha) \leq f_{\bm{p}}(\alpha) \leq \vartheta_{\bm{p}}^*(\alpha).
    \end{equation}
    Taking concave conjugates in \cref{e:gen} and recalling that concave conjugates reverse order,
    \begin{equation}\label{e:upper}
        \vartheta_{\bm{p}}^{**}(q) \geq \vartheta_{\bm{p}}(q) \geq V_{\bm{p}}(q) = F_{\bm{p}}^*(q) \geq f_{\bm{p}}^*(q) \geq\vartheta_{\bm{p}}^{**}(q).
    \end{equation}
    Therefore equality holds everywhere, as claimed.
\end{proof}
\begin{remark}
    The first inequality in \cref{e:upper} was required since \emph{a priori} $\vartheta_{\bm{p}}$ need not be concave; see \cite[Section~3]{zbl:0841.28012}.
    Regardless, we recall that $g^{**}$ is the concave hull of $g$ for general functions $g$, so it always holds that $g^{**} \geq g$ which turns out to be sufficient.
\end{remark}
To conclude this section, let us also point out that $F_{\bm{p}}$ can be written as an optimization over $\eta(\mathcal{P})$.
Let $T_{\bm{p}}^*(\bm{v},\rho)$ denote the concave conjugate of $T_{\bm{p}}(\bm{v}, q)$ in the $q$-variable.
The parameter $\rho$ will be related to $\alpha$ by the following coordinate change:
\begin{equation*}
    \rho(\alpha, \bm{v}) = \alpha - \frac{\CH{\bm{v}}{\eta(\bm{p})}}{\chi_1(\bm{v})}.
\end{equation*}
Finally, recall from \cref{c:col-dual} that
\begin{equation*}
    T_{\bm{p}}^*(\bm{v},\rho) = \max_{\bm{w}\in \eta^{-1}(\bm{v})}
        \left\{
            A(\bm{w}, \bm{v}): B_{\bm{p}}(\bm{w},\bm{v}) = \rho.
        \right\}
\end{equation*}
We now have the following.
\begin{proposition}
    We have
    \begin{equation*}
        F_{\bm{p}}(\alpha)
        =
        \max_{\bm{v}\in\eta(\mathcal{P})}
        \left\{
            \frac{H(\bm{v})}{\chi_1(\bm{v})}
            +
            T^*_{\bm{p}}\left(\bm{v}, \alpha-\frac{\CH{\bm{v}}{\eta(\bm{p})}}{\chi_1(\bm{v})}\right)
        \right\}.
    \end{equation*}
\end{proposition}
\begin{proof}
    We compute
    \begin{align*}
        F_{\bm{p}}(\alpha) &= \max_{\bm{w}\in\mathcal{P}} \left\{v(\bm{w}):u_{\bm{p}}(\bm{w}) = \alpha\right\}\\
                           &= \max_{\bm{v}\in\eta(\mathcal{P})}\max_{\bm{w}\in\eta^{-1}(\bm{v})}\left\{A(\bm{w}, \bm{v}) +\frac{H(\bm{v})}{\chi_1(\bm{v})} :B_{\bm{p}}(\bm{w}, \bm{v}) = \alpha -\frac{\CH{\bm{v}}{\eta(\bm{p})}}{\chi_1(\bm{v})} \right\}\\
                           &= \max_{\bm{v}\in\eta(\mathcal{P})}\left(\frac{H(\bm{v})}{\chi_1(\bm{v})} + \max_{\bm{w}\in\eta^{-1}(\bm{v})}\left\{A(\bm{w}, \bm{v}) :B_{\bm{p}}(\bm{w}, \bm{v}) = \rho(\alpha, \bm{v}) \right\}\right)\\
                           &=
        \max_{\bm{v}\in\eta(\mathcal{P})}
        \left\{
            \frac{H(\bm{v})}{\chi_1(\bm{v})}
            +
            T^*_{\bm{p}}\left(\bm{v}, \rho(\alpha, \bm{v})\right)
        \right\}
    \end{align*}
    as claimed.
\end{proof}

\subsection{Optimizing with constrained eccentricity}\label{ss:const}
In this section, we note an additional result which may help to clarify why Gatzouras--Lalley carpets behave differently than Bedford--McMullen carpets.
This result will also be relevant in \cref{s:two-col}.

For $\bm{w}\in\mathcal{P}$, write
\begin{equation*}
    \Gamma(\bm{w}) = \frac{\chi_1(\bm{w})}{\chi_2(\bm{w})}.
\end{equation*}
Note that $\Gamma$ takes values in a compact interval
\begin{equation*}
    \Gamma(\mathcal{P}) = [\kappa_{\min},\kappa_{\max}]\subset (0,1)
\end{equation*}
where
\begin{equation*}
    \kappa_{\min}\coloneqq \min_{i\in\mathcal{I}}\frac{\log\beta_{i,1}}{\log\beta_{i,2}}\qquad\text{and} \qquad\kappa_{\max}\coloneqq \max_{i\in\mathcal{I}}\frac{\log\beta_{i,1}}{\log\beta_{i,2}}.
\end{equation*}
We say that the IFS and its associated carpet are \emph{homogeneous} if $\kappa_{\min} = \kappa_{\max}$.

We constrain the optimization problems defining $F_{\bm{p}}$ and $V_{\bm{p}}$ to those probability vectors with fixed eccentricity.
To be precise, for $\kappa\in\Gamma(\mathcal{P})$, write
\begin{align*}
    V_{\bm{p},\kappa}(q)&\coloneqq\min_{\bm{w}\in\Gamma^{-1}(\kappa)}\left\{q\cdot u_{\bm{p}}(\bm{w}) - v(\bm{w})\right\}\\
    F_{\bm{p},\kappa}(\alpha)&\coloneqq\max_{\bm{w}\in\Gamma^{-1}(\kappa)}\left\{v(\bm{w}):u_{\bm{p}}(\bm{w}) = \alpha\right\}
\end{align*}
These constrained optimization problems are always dual problems.
\begin{proposition}
    For each $\kappa\in\Gamma(\mathcal{P})$ and $q\in\R$, there is a unique $\bm{w}\in\mathcal{P}$ with $\Gamma(\bm{w}) = \kappa$ such that
    \begin{equation*}
        V_{\bm{p},\kappa}(q) = q\cdot u_{\bm{p}}(\bm{w}) - v(\bm{w}).
    \end{equation*}
    In particular, duality holds and
    \begin{equation*}
        F_{\bm{p},\kappa}(\alpha) = V_{\bm{p},\kappa}^*(\alpha).
    \end{equation*}
\end{proposition}
\begin{proof}
    For $\kappa\in\Gamma(\mathcal{P})$ fixed, write
    \begin{align*}
        u_{\bm{p},\kappa}(\bm{w}) &=\frac{\kappa \CH{\bm{w}}{\bm{p}} + (1-\kappa)\CH{\eta(\bm{w})}{\eta(\bm{p})}}{\chi_1(\eta(\bm{w}))},\\
        v_{\kappa}(\bm{w}) &=\frac{\kappa H(\bm{w}) + (1-\kappa)H(\eta(\bm{w}))}{\chi_1(\eta(\bm{w}))}.
    \end{align*}
    Then, a direct computation gives for $\bm{w}\in\Gamma^{-1}(\kappa)$ that
    \begin{align*}
        q\cdot u_{\bm{p}}(\bm{w}) - v(\bm{w}) &=q\cdot u_{\bm{p},\kappa}(\bm{w}) - v_\kappa(\bm{w})
    \end{align*}
    But $\Gamma^{-1}(\kappa)$ is a convex subset of $\mathcal{P}$ and $q\cdot u_{\bm{p},\kappa}-v_\kappa$ has a strictly convex numerator and a positive affine denominator, so it is strictly quasi-convex by \cref{l:qc-const}.
    The conclusion follows by \cref{l:qc-min}.
\end{proof}
In particular, we can obtain a result which is a mild generalization of a theorem of King for Bedford--McMullen carpets \cite{zbl:0845.28007}.
\begin{corollary}\label{c:homogeneous}
    Let $\mu_{\bm{p}}$ be a Gatzouras--Lalley measure with separation of principal projections and whose defining IFS is homogeneous.
    Then $\vartheta_{\bm{p}}$ is differentiable and the fine multifractal formalism holds.
\end{corollary}

\subsection{Equality of the fine and coarse \texorpdfstring{$L^q$}{Lq}-spectra}\label{ss:lq-equal}
In this section, we characterize when the fine $L^q$-spectrum $\vartheta_{\bm{p}}$ and the coarse $L^q$-spectrum $\tau_{\bm{p}}$ are equal.

Let $Y_{\eta(\bm{p})}(q)$ denote $L^q$-spectrum of $\eta_*\mu_{\bm{p}}$: that is,
\begin{equation*}
    \sum_{\ell\in\eta(\mathcal{I})}\eta(\bm{p})^q_\ell \ctr_{i,1}^{-Y_{\eta(\bm{p})}(q)} = 1.
\end{equation*}
We first recall a special case of \cite[Theorem~2]{zbl:1091.28005}.
\begin{proposition}[\cite{zbl:1091.28005}]\label{p:lq-form}
    Let $\mu_{\bm{p}}$ be a Gatzouras--Lalley measure and $q\in\R$.
    Then $\tau_{\bm{p}}(q)$ is the unique solution to the equation
    \begin{equation*}
        \sum_{i\in\mathcal{I}}p_i^q\ctr_{\eta(i),1}^{-Y_{\eta(\bm{p})}(q)}\ctr_{i, 2}^{-(\tau_{\bm{p}}(q) - Y_{\eta(\bm{p})}(q))} = 1.
    \end{equation*}
\end{proposition}
Now, in order to state our characterization, we need some more notation.
For $\ell\in\eta(\mathcal{I})$ and $\bm{p}\in\mathcal{P}$, let $t_{\bm{p},\ell}(q)$ be the unique solution to the equation
\begin{equation*}
    \sum_{j\in\eta^{-1}(\ell)}\left(\frac{\bm{p}_j}{\eta(\bm{p})_\ell}\right)^q \ctr_{i,2}^{-t_{\bm{p},\ell}(q)} = 1.
\end{equation*}
Precisely, $t_{\bm{p},\ell}$ is the $L^q$ spectrum of the self-similar IFS corresponding to column $\ell$.
Equivalently, $t_{\bm{p},\ell} = T_{\bm{p}}(\bm{\delta}_\ell, q)$ where $\bm{\delta}_\ell$ is the dirac mass on column $\ell$.
\begin{corollary}\label{c:lq-equal}
    Let $\mu_{\bm{p}}$ be a Gatzouras--Lalley measure with separation of principal projections and let $q\in\R$.
    Then the following are equivalent:
    \begin{enumerate}[nl,r]
        \item\label{i:t-1} $\tau_{\bm{p}}(q) = \vartheta_{\bm{p}}(q)$.
        \item\label{i:t-2} The function $\ell\mapsto t_{\bm{p},\ell}(q)$ is constant.
        \item\label{i:t-3} The function $\bm{v}\mapsto T_{\bm{p}}(\bm{v}, q)$ is constant.
    \end{enumerate}
    Moreover, if any of the above equivalent conditions hold, writing $t_{\bm{p}}(q) = t_{\bm{p},\ell}(q)$ for the constant value of $\ell\mapsto t_{\bm{p},\ell}(q)$,
    \begin{equation*}
        \tau_{\bm{p}}(q) = \vartheta_{\bm{p}}(q) = Y_{\eta(\bm{p})}(q) + t_{\bm{p}}(q).
    \end{equation*}
\end{corollary}
\begin{proof}
    Clearly \cref{i:t-3} implies \cref{i:t-2}.
    For the converse, assuming \cref{i:t-2} and writing $t_{\bm{p}}(q)$ for the common value, we know that $t_{\bm{p}}(q) = T_{\bm{p}}(\bm{\delta}_\ell, q)$ for all $\ell\in\eta(\mathcal{I})$.
    Therefore if $\bm{v}\in\eta(\mathcal{P})$ is arbitrary,
    \begin{align*}
        \Psi(\bm{v}, q, t_{\bm{p}}(q)) &= -\sum_{\ell\in\eta(\mathcal{I})}\bm{v}_\ell \log\left(\sum_{j\in\eta^{-1}(\ell)}p_j^q\beta_{j,2}^{-t_{\bm{p}}(q)}\right)\\
                                       &= -\sum_{\ell\in\eta(\mathcal{I})}\bm{v}_\ell \log\eta(\bm{p})_\ell^q\\
                                       &= q \CH{\bm{v}}{\eta(\bm{p})}.
    \end{align*}
    Therefore $T_{\bm{p}}(\bm{v}, q) = t_{\bm{p}}(q)$ by uniqueness.

    Next we prove that \cref{i:t-1} and \cref{i:t-2} are equivalent.
    Define probability vectors $\bm{y}\in\eta(\mathcal{P})$ and $\bm{z}\in\mathcal{P}$ by the rule
    \begin{equation*}
        \bm{y} = \bigl(\eta(\bm{p})_\ell^q \ctr_{\ell,1}^{-Y_{\eta(\bm{p})}(q)}\bigr)_{\ell\in\eta(\mathcal{I})},\qquad \bm{z} = \bigl(\bm{p}_i^q \ctr_{\eta(i), 1}^{-Y_{\eta(\bm{p})}(q)}\ctr_{i,2}^{-(\tau_{\bm{p}}(q) - Y_{\eta(\bm{p})}(q))}\bigr)_{i\in\mathcal{I}}.
    \end{equation*}
    (These are the minimizing probability vectors for the optimization stated in \cite[Theorem~2.1]{zbl:1549.37013} in the special case of Gatzouras--Lalley measures.)
    A direct computation gives for $\bm{w}\in\mathcal{P}$ that
    \begin{equation*}
        q u_{\bm{p}}(\bm{w})-v(\bm{w})-\tau_{\bm{p}}(q) = \frac{\DKL{\bm{w}}{\bm{z}}}{\chi_2(\bm{w})} + \left(\frac{1}{\chi_1(\eta(\bm{w}))}-\frac{1}{\chi_2(\bm{w})}\right)\DKL{\eta(\bm{w})}{\bm{y}}.
    \end{equation*}
    Taking the minimum over $\bm{w}$ and recalling \cref{it:main}, we see that $\vartheta_{\bm{p}}(q) = \tau_{\bm{p}}(q)$ if and only if $\eta(\bm{z}) = \bm{y}$, in which case the minimum is attained exactly at $\bm{z}$.
    Now,
    \begin{equation*}
        \eta(\bm{z})_\ell = \eta(\bm{p})_\ell^q \ctr_{\ell, 1}^{-Y_{\eta(\bm{p})}(q)}\sum_{i\in\eta^{-1}(\ell)}\left(\frac{\bm{p}_i}{\eta(\bm{p})_\ell}\right)^q\ctr_{i,2}^{-(\tau_{\bm{p}}(q) - Y_{\eta(\bm{p})}(q))}.
    \end{equation*}
    Term $\ell$ in the summation on the right is exactly 1 if and only if $\tau_{\bm{p}}(q) =Y_{\eta(\bm{p})}(q)+t_{\bm{p},\ell}(q)$.
    Therefore comparing $\eta(\bm{z})_\ell$ to the definition of $\bm{y}$, we see that $\eta(\bm{z})=\bm{y}$ if and only if \cref{i:t-2} holds.
    Thus \cref{i:t-1} and \cref{i:t-3} are equivalent.

    To conclude, we also see from the above that $\tau_{\bm{p}}(q) = Y_{\eta(\bm{p})}(q)+t_{\bm{p}}(q)$, as claimed.
\end{proof}
We can now complete the proof of \cref{ic:coarse-formalism}.
\begin{proofref}{ic:coarse-formalism}
    Equivalence of \cref{i:cf-2} and \cref{i:cf-3} is simply \cref{c:lq-equal}.

    Assuming \cref{i:cf-2}, since $\tau_{\bm{p}}$ is differentiable, \cref{i:cf-1} follows using \cref{ic:diff}.
    Conversely, assuming \cref{i:cf-2}, it follows that $\vartheta_{\bm{p}} = f_{\bm{p}}^* = \tau_{\bm{p}}$, which is \cref{i:cf-1}.
\end{proofref}

\section{Gatzouras--Lalley measures with two columns}\label{s:two-col}
We now turn our attention to Gatzouras--Lalley measures with two columns.
We will prove our variational formula for the multifractal spectrum \cref{it:two-col-var} and construct some interesting examples to prove \cref{it:two-col-exc}.

\subsection{A variational formula for the multifractal spectrum for systems with two columns}
Recall in \cref{p:mf-lower} that we noted a general lower bound for the multifractal spectrum.
Unfortunately, we are not able to determine in general if this inequality is necessarily an equality, or if the inequality may in some cases be strict.
(Of course, equality always holds at slopes corresponding to points of differentiability of $\vartheta_{\bm{p}}(q)$.)

Why does a similar approach from our proof of the variational formula for the fine $L^q$-spectrum not also handle the multifractal spectrum?
The proof of \cref{c:lq-lower} shows, in general, that it is always possible to choose a good measure for a given point $x\in\Lambda$ using a well-chosen sequence of empirical frequency measures, depending on the coding of the point $x$.
The problem is that such a measure is only valid along an infinite sequence of scales, and at those scales the local dimension might not satisfy the additional constraint required by the definition of $F_{\bm{p}}$.

However, when $\mu_{\bm{p}}$ is a Gatzouras--Lalley measure with two columns (so that $\eta(\mathcal{P})$ is a 1-dimensional space), we can show that equality indeed holds.
Before we can prove our result, we need a continuity result about the map $k\mapsto \bm{\lambda}_k$.
Here, we embed $\mathcal{P}\subset\R^{\#\mathcal{I}}$ and use the supremum norm.
\begin{lemma}\label{l:cont}
    For all $\gamma\in\mathcal{I}^{\N}$ and $k,m\in\N$,
    \begin{equation*}
        \norm{\bm{\lambda}_k(\gamma)-\bm{\lambda}_{k+m}(\gamma)}_\infty \leq \frac{2m}{k}.
    \end{equation*}
\end{lemma}
\begin{proof}
    Write $\gamma = (i_n)_{n=1}^\infty$.
    For notational simplicity, we suppress $\gamma$ for the remainder of the proof.
    Then, by definition,
    \begin{equation*}
        (k+m)\cdot \bm{\lambda}_{k+m} - k\cdot \bm{\lambda}_k = \bigl(\#\{n:k+1 \leq n \leq k + m,\ i_n = j\}\bigr)_{j\in\mathcal{I}} \in [0,m]^{\#\mathcal{I}}
    \end{equation*}
    so
    \begin{equation*}
        \norm{\bm{\lambda}_k - \bm{\lambda}_{k+m}}_\infty
        \leq \norm{\bm{\lambda}_k - \tfrac{k+m}{k}\bm{\lambda}_{k+m}}_\infty + \norm{\tfrac{k+m}{k}\bm{\lambda}_{k+m}-\bm{\lambda}_{k+m}}_\infty\\
        \leq \frac{m}{k}+\frac{m}{k}
    \end{equation*}
    as claimed.
\end{proof}
Now, using the assumption that $\#\eta(\mathcal{I}) = 2$, we can use \cref{l:cont} to choose a good sequence of scales.
\begin{lemma}\label{l:ivt}
    Suppose $\#\eta(\mathcal{I}) = 2$ and let $\gamma\in\mathcal{I}^{\N}$ be arbitrary.
    Then there exists a subsequence $(k_n)_{n=1}^\infty$ such that
    \begin{equation*}
        \lim_{n\to\infty}\eta(\bm{\lambda}_{k_n}(\gamma)) = \lim_{n\to\infty}\eta(\bm{\lambda}_{L_{k_n}(\gamma)}(\gamma)).
    \end{equation*}
\end{lemma}
\begin{proof}
    Using \cref{l:cont}, this is essentially the intermediate value theorem.
    First, since the eccentricity of the rectangles in the IFS are bounded strictly between $0$ and $1$, there is a constant $m$ such that $0 \leq L_{n+1}(\gamma) - L_n(\gamma) \leq m$ for all $n\in\N$.

    Since $\#\eta(\mathcal{I}) = 2$, consider the homeomorphism $\iota\colon\eta(\mathcal{P}) \to [0,1]$ given by $\iota((p, 1-p)) = p$ so $\norm{\bm{v} - \bm{w}}_\infty = \norm{\iota(\bm{v}) - \iota(\bm{w})}_\infty$.
    Write
    \begin{equation*}
        \alpha_n = \iota(\eta(\bm{\lambda}_{L_n(\gamma)}(\gamma))) - \iota(\eta(\bm{\lambda}_n(\gamma))) \in [-1, 1].
    \end{equation*}
    By considering subsequences along which $\iota(\eta(\bm{\lambda}_n(\gamma)))$ is maximal or minimal, we see that $\liminf_{n\to\infty}\alpha_n \leq 0 \leq \limsup_{n\to\infty}\alpha_n$.
    But $|\alpha_{n+1} - \alpha_n| \leq 4m/n$ by \cref{l:cont}, where $m$ is fixed, so there exists a subsequence $(k_n)_{n=1}^\infty$ such that $\lim_{n\to\infty}\alpha_n = 0$.
    Passing to a subsequence again if necessary, the proof is complete.
\end{proof}
Now we can prove the variational principle for the multifractal spectrum.
\begin{restatement}{it:two-col-var}
    Let $\mu_{\bm{p}}$ be a Gatzouras--Lalley measure with separation of principal projections and with two columns.
    Then for all $\alpha\in\R$,
    \begin{align*}
        f_{\bm{p}}(\alpha) = F_{\bm{p}}(\alpha) &= \sup_{\bm{w}\in\mathcal{P}}\left\{v(\bm{w}):u_{\bm{p}}(\bm{w}) = \alpha\right\}\\*
        &=
        \max_{\bm{v}\in\eta(\mathcal{P})}
        \left\{
            \frac{H(\bm{v})}{\chi_1(\bm{v})}
            +
            T^*_{\bm{p}}\left(\bm{v}, \alpha-\frac{\CH{\bm{v}}{\eta(\bm{p})}}{\chi_1(\bm{v})}\right)
        \right\}.
    \end{align*}
\end{restatement}
\begin{proof}
    Recall from \cref{p:mf-lower} that the lower bound $f_{\bm{p}} \geq F_{\bm{p}}$ always holds.
    Fix $\alpha\in\R$ and set
    \begin{equation*}
        E_\alpha=\{x\in\Lambda:\dim_{\loc}(\mu_{\bm{p}},x)=\alpha\}.
    \end{equation*}

    Fix $\delta>0$.
    To obtain the upper bound $\dimH E_\alpha\leq F_{\bm{p}}(\alpha)$, it suffices to show for all $x\in E_\alpha$ that there are $\bm{w}\in\mathcal{P}$ and a fully supported $\bm{z}\in\mathcal{P}$ such that $u_{\bm{p}}(\bm{w})=\alpha$ and
    \begin{equation*}
        \liminf_{r \to 0}\frac{\log \mu_{\bm{z}}\bigl(B(x,r)\bigr)}{\log r} \leq v(\bm{w})+\delta.
    \end{equation*}
    If the constraint $u_{\bm{p}}(\bm{w})=\alpha$ has no solution, then $E_\alpha=\varnothing$.
    Otherwise, $v(\bm{w})\leq F_{\bm{p}}(\alpha)$, so \cref{p:pointwise-var} with $\lambda=\mu_{\bm{p}}$, $E=E_\alpha$, $q=0$, $t=F_{\bm{p}}(\alpha)+\delta$, and $\Delta=\{\mu_{\bm{z}}:\bm{z}\in\mathcal{P}\}$ gives $\dimH E_\alpha\leq F_{\bm{p}}(\alpha)+\delta$.
    Here, we recall that $\Delta$ has uniform densities by \cref{l:meas-quasi-compact}.
    Taking $\delta\to0$ will give the upper bound.

    Let $x = \pi(\gamma)\in E_\alpha$.
    By \cref{l:square-measure} and \cref{l:spsc-reduction},
    \begin{equation*}
        \lim_{k\to\infty}
        \left(\frac{\CH{\eta(\bm{\lambda}_{L_k(\gamma)})}{\eta(\bm{p})}}{\chi_1(\eta(\bm{\lambda}_{L_k(\gamma)}))}
        +\frac{\CH{\bm{\lambda}_k}{\bm{p}}-\CH{\eta(\bm{\lambda}_k)}{\eta(\bm{p})}}{\chi_2(\bm{\lambda}_k)}\right) = \alpha.
    \end{equation*}
    Since $\#\eta(\mathcal{I}) = 2$, by \cref{l:ivt}, there exists a subsequence $(k_n)_{n=1}^\infty$ such that
    \begin{equation*}
        \lim_{n \to\infty}\eta(\bm{\lambda}_{L_{k_n}}(\gamma)) = \lim_{n\to\infty}\eta(\bm{\lambda}_{k_n}) \eqqcolon \bm{v}.
    \end{equation*}
    Passing to a further subsequence if necessary, we may also assume that $\bm{w} = \lim_{n\to\infty}\bm{\lambda}_{k_n}$.
    Since $\eta(\bm{w}) = \bm{v}$,
    \begin{equation*}
        u_{\bm{p}}(\bm{w}) = \frac{\CH{\eta(\bm{w})}{\eta(\bm{p})}}{\chi_1(\bm{v})}
        +\frac{\CH{\bm{w}}{\bm{p}}-\CH{\eta(\bm{w})}{\eta(\bm{p})}}{\chi_2(\bm{w})} = \alpha.
    \end{equation*}
    For $\varepsilon\in(0,1)$, set $\bm{w}_\varepsilon=(1-\varepsilon)\bm{w}+\varepsilon\bm{p}$, which is fully supported.
    Applying \cref{l:square-measure} and \cref{l:spsc-reduction} along the subsequence,
    \begin{equation*}
        \liminf_{r\to0}\frac{\log\mu_{\bm{w}_\varepsilon}\bigl(B(x,r)\bigr)}{\log r}
        \leq\frac{\CH{\bm{v}}{\eta(\bm{w}_\varepsilon)}}{\chi_1(\bm{v})}
        +\frac{\CH{\bm{w}}{\bm{w}_\varepsilon}-\CH{\bm{v}}{\eta(\bm{w}_\varepsilon)}}{\chi_2(\bm{w})}.
    \end{equation*}
    As $\varepsilon\to0$, the right-hand side tends to $v(\bm{w})$.
    Taking $\varepsilon>0$ sufficiently small and $\bm{z}=\bm{w}_\varepsilon$ proves the required bound.
\end{proof}

\subsection{Gatzouras--Lalley measures with homogeneous columns}\label{ss:two-col-special}
All of our examples are constructed using a special family of Gatzouras--Lalley carpets with two columns.
These examples will fall in the following special class.
\begin{definition}
    We say that a Gatzouras--Lalley IFS and its associated carpet have \emph{homogeneous columns} if for all $k\in\eta(\mathcal{I})$ and $j\in\{1,2\}$ the map $i\mapsto\ctr_{i,j}$ is constant on the column $\eta^{-1}(k)$.
    Similarly, we say that a Gatzouras--Lalley measure $\mu_{\bm{p}}$ for $\bm{p}\in\mathcal{P}(\mathcal{I})$ has \emph{homogeneous columns} if the underlying IFS has homogeneous columns and moreover for every $k\in\eta(\mathcal{I})$ the map $i\mapsto\bm{p}_i$ is constant on the column $\eta^{-1}(k)$.
\end{definition}
When the IFS has homogeneous columns, we write $\ctr_{i,1} = a_{\eta(i)}$ and $\ctr_{i,2} = b_{\eta(i)}$ for numbers $0<b_j < a_j < 1$ where $j\in\eta(\mathcal{I})$.
We also set $\ell_j = \#\eta^{-1}(j)$ for $j\in\eta(\mathcal{I})$.
For a measure with homogeneous columns, we also abuse notation and use $\bm{p}$ and $\eta(\bm{p})$ interchangeably.

For $\bm{v},\bm{p}\in\eta(\mathcal{P})$, we use some shorthand:
\begin{align*}
    \chi_2(\bm{v})&=\sum_{j\in\eta(\mathcal{I})}v_j\log(1/b_j),\\
    \ell(\bm{v})&=\sum_{j\in\eta(\mathcal{I})}v_j \log \ell_j.
\end{align*}
When $\mu_{\bm{p}}$ has homogeneous columns, for $\bm{v}\in\eta(\mathcal{P})$,
\begin{align*}
    T_{\bm{p}}(\bm{v},q) =\frac{(q-1)\ell(\bm{v})}{\chi_2(\bm{v})}
\end{align*}
and therefore
\begin{align*}
    \vartheta_{\bm{p}}(q) &= \min_{\bm{v}\in\eta(\mathcal{P})} \left\{ \frac{q \CH{\bm{v}}{\bm{p}}-H(\bm{v})}{\chi_1(\bm{v})} + T_{\bm{p}}(\bm{v},q) \right\}.
\end{align*}

Now, consider the special case that the IFS in addition has two columns.
We reparametrize $\eta(\mathcal{P})$ as $[0,1]$ through the map $v\mapsto \bm{v} \coloneqq (v, 1-v)$.
In this case, we know that
\begin{equation*}
    f_{\bm{p}}(\alpha) = \max_{v\in[0,1]} \left\{ \frac{H(\bm{v})}{\chi_1(\bm{v})} + T^*_{\bm{p}}\left(\bm{v},\alpha-\frac{\CH{\bm{v}}{\bm{p}}}{\chi_1(\bm{v})}\right) \right\}.
\end{equation*}
Since $T_{\bm{p}}(\bm{v}, q)$ is linear in $q$,
\begin{equation*}
    T^*_{\bm{p}}(\bm{v}, \rho) = \begin{cases}
        \frac{\ell(\bm{v})}{\chi_2(\bm{v})} &: \rho = \frac{\ell(\bm{v})}{\chi_2(\bm{v})},\\
        -\infty &: \text{ otherwise}.
    \end{cases}
\end{equation*}
Alternatively, write
\begin{equation}\label{e:AD-def}
    A_{\bm{p}}(\bm{v}) = \frac{\CH{\bm{v}}{\bm{p}}}{\chi_1(\bm{v})} + \frac{\ell(\bm{v})}{\chi_2(\bm{v})}\qquad D(\bm{v})=\frac{H(\bm{v})}{\chi_1(\bm{v})} + \frac{\ell(\bm{v})}{\chi_2(\bm{v})}.
\end{equation}
Then, we equivalently have
\begin{align*}
    \vartheta_{\bm{p}}(q) &= \min_{v\in[0,1]} \left\{ q\cdot A_{\bm{p}}(\bm{v}) - D(\bm{v})\right\},\\
    f_{\bm{p}}(\alpha) &= \max_{v\in[0,1]} \left\{ D(\bm{v}): A_{\bm{p}}(\bm{v}) = \alpha\right\}.
\end{align*}
In particular, $f_{\bm{p}}(\alpha)$ is the upper hull of the continuous curve $v\mapsto \bigl(A_{\bm{p}}(\bm{v}), D(\bm{v})\bigr)$ for $v\in[0,1]$.
See \cref{f:exp} for a depiction of the function $v\mapsto\bigl(A_{\bm{p}}(\bm{v}), D(\bm{v})\bigr)$.

Let us introduce some more notation for working with Gatzouras--Lalley measures with two homogeneous columns.
First, abbreviate $A_p(v)=A_{\bm{p}}(\bm{v})$ and $D(v)=D(\bm{v})$, where $\bm{p}=(p,1-p)$ and $\bm{v}=(v,1-v)$.
Here is some more shorthand:
\begin{enumerate}[nl]
    \item Write $\eta(\mathcal{I}) = \{1,2\}$.
    \item Let $I_p = A_p([0,1])= [\alpha_{\min}(p),\alpha_{\max}(p)]$ denote the support of $f_{\bm{p}}$.
    \item For $j\in\{1,2\}$, let $\phi_j = \frac{\log\ell_j}{\log(1/b_j)}$ denote the dimension of column $j$.
    \item For each $\alpha\in I_p$, define
        \begin{equation*}
            Q_\alpha(v)\coloneqq\chi_1(\bm{v})\chi_2(\bm{v})\bigl(A_p(v)-\alpha\bigr),
        \end{equation*}
        which is a polynomial of degree at most two.
    \item Let
        \begin{equation*}
            \omega \coloneqq \log(1/b_1)\log(1/a_2)-\log(1/b_2)\log(1/a_1)
        \end{equation*}
        measure the difference in eccentricity between $0$ and $1$.
\end{enumerate}
We note the following degeneracy property of the multifractal spectrum.
\begin{lemma}\label{l:const-char}
    Let $\mu_{\bm{p}}$ be a Gatzouras--Lalley measure with two homogeneous columns.
    The following are equivalent:
    \begin{enumerate}[nl,a]
        \item\label{i:a1} $I_p$ is not a singleton.
        \item\label{i:a2} $\#A_p^{-1}(\alpha) \leq 2$ for all $\alpha\in I_p$.
        \item\label{i:a3} either $A_p(0) \neq A_p(1)$, or $A_p(0) = A_p(1)$, $\phi_1 \neq \phi_2$, and $\omega \neq 0$.
    \end{enumerate}
\end{lemma}
\begin{proof}
    Since $\chi_1(\bm{v})\chi_2(\bm{v})$ is positive, either $Q_\alpha(v)$ vanishes identically so $A_p\equiv \alpha$, or there are at most two solutions to $Q_\alpha(v) = 0$.
    This shows that \cref{i:a1} and \cref{i:a2} are equivalent.

    Next, if $A_p(0) = A_p(1) \eqqcolon \alpha_0$, then
    \begin{equation*}
        Q_{\alpha_0}(v) = (\phi_1-\phi_2)\cdot\omega\cdot v(1-v)
    \end{equation*}
    which vanishes if and only if $\phi_1=\phi_2$ or $\omega =0$.
    Thus \cref{i:a1} and \cref{i:a3} are equivalent.
\end{proof}
Finally, we point out that the function $A_p$ is necessarily either quasi-concave or quasi-convex.
This follows immediately from the following slightly more general result.
\begin{lemma}
    Let $W\subset\R$ be a compact interval and suppose $f\colon W \to\R$ is a function of the form
    \begin{equation*}
        f = \frac{g_1}{h_1} + \frac{g_2}{h_2}
    \end{equation*}
    where the functions $g_i, h_i$ are affine and the functions $h_i$ are strictly positive on $W$.
    Then $f$ is either quasi-concave or quasi-convex.
\end{lemma}
\begin{proof}
    If $h_1 = c h_2$ for some constant $c > 0$, this result follows from \cref{l:qc-const} (in fact, $f$ is both quasi-concave and quasi-convex).
    Otherwise, $\{h_1,h_2\}$ spans the space of all affine functions, so with $\psi=h_1/h_2$ we can write
    \begin{equation*}
        f = \xi\circ\psi\qquad\text{where}\qquad \xi(p)=\frac{ap^2+bp+c}{p},\quad p>0,
    \end{equation*}
    for some $a,b,c\in\R$.
    By \cref{l:qc-const}, $\xi$ is quasi-convex if $a\geq 0$ and quasi-concave if $a\leq 0$.
    Since $\psi$ is monotone, the same holds for $f=\xi\circ\psi$.
\end{proof}
\begin{remark}\label{r:mini}
    If $A_p$ is non-degenerate (i.e.\ it satisfies one of the equivalent conditions in \cref{l:const-char}), it must be strictly quasi-convex or strictly quasi-concave, which means that one of the global maximum or the global minimum must be attained at the boundary (i.e.\ at $0$ or $1$).
\end{remark}

\subsection{Discontinuity of the multifractal spectrum}
In this section, we construct examples of Gatzouras--Lalley measures with discontinuous multifractal spectrum.

We begin with a general characterization of (dis)continuity of the multifractal spectrum.
\begin{lemma}\label{p:disc}
    Let $\mu_{\bm{p}}$ be a Gatzouras--Lalley measure with separation of principal projections and with two homogeneous columns.
    Let $\alpha_0 \in I_p$.
    Then $f_{\bm{p}}|_{I_p}$ is discontinuous at $\alpha_0$ if and only if $A_p^{-1}(\alpha_0) = \{v, z\}$ for some $z\in\{0,1\}$ and $v\in(0,1)$ such that $D(z) > D(v)$.
\end{lemma}
\begin{proof}
    If $I_p$ is a singleton or $A_p$ is strictly monotone, the assertion is immediate.
    Otherwise, by \cref{l:const-char,r:mini}, $A_p$ attains a global minimum or maximum at some $c\in(0,1)$, with strictly monotone restrictions to $[0,c]$ and $[c,1]$.

    Let $J_1,J_2$ be the images $A_p([0,c])$ and $A_p([c,1])$, labelled so that $J_1\subseteq J_2=I_p$, and let $v_i\colon J_i\to[0,1]$ be the corresponding continuous inverse branches.
    The intervals $J_1,J_2$ share the endpoint $A_p(c)$.
    Define the continuous functions $d_i=D\circ v_i$ on $J_i$.
    Then for $\alpha\in I_p$,
    \begin{equation*}
        f_{\bm{p}}(\alpha)=
        \begin{cases}
            \max\{d_1(\alpha),d_2(\alpha)\} &: \alpha\in J_1,\\
            d_2(\alpha) &: \alpha\in I_p\setminus J_1.
        \end{cases}
    \end{equation*}
    Thus $f_{\bm{p}}|_{I_p}$ is continuous except possibly at an endpoint $\alpha_0$ of $J_1$ lying in the interior of $I_p = J_2$.
    Precisely, $A_p^{-1}(\alpha_0)=\{z,v\}$ for $z\in\{0,1\}$ and $v\in(0,1)$: the first branch ends at $z=v_1(\alpha_0)$ while the second continues through $v=v_2(\alpha_0)$.
    The one-sided limits are
    \begin{equation*}
        \max\{D(z),D(v)\}=f_{\bm{p}}(\alpha_0) \qquad\text{and}\qquad D(v).
    \end{equation*}
    Hence discontinuity occurs exactly when $D(z)>D(v)$.
\end{proof}
Of course, \cref{p:disc} is not sufficient to give any concrete examples of Gatzouras--Lalley measures with discontinuous multifractal spectrum.
In fact, there are two cases where it is easy to see that $\vartheta_{\bm{p}}$ is differentiable and the multifractal formalism holds: firstly, if the eccentricity
$\chi_1(\bm{w})/\chi_2(\bm{w})$ is constant (equivalently, $\omega = 0$) so we may apply \cref{c:homogeneous}; and secondly, if $\phi_1 = \phi_2\eqqcolon \phi$ in which case $t_{\bm{p},1}(q) = t_{\bm{p},2}(q) = \phi(q-1)$ so we may apply \cref{c:lq-equal}.

It turns out that if neither of these two conditions hold, then there is always an open set of probability vectors in $\eta(\mathcal{P})$ for which the multifractal formalism fails.
\begin{theorem}\label{t:exc}
    Let $\mathcal{G}$ be a Gatzouras--Lalley carpet with separation of principal projections and with two homogeneous columns.
    Suppose $\mathcal{G}$ is not homogeneous and $\phi_1 \neq \phi_2$.
    Then, there exists a non-empty open subset $U\subset\eta(\mathcal{P})$ such that $f_{\bm{p}}$ has a jump discontinuity in the interior of its support for all $\bm{p}\in U$.
\end{theorem}
\begin{proof}
    Exchanging the columns if necessary, assume that $D(1)=\phi_1>\phi_2=D(0)$.
    The function
    \begin{equation*}
        \Delta(p)\coloneqq A_p(1)-A_p(0)=-\frac{\log p}{\log(1/a_1)}+\frac{\log(1-p)}{\log(1/a_2)}+\phi_1-\phi_2
    \end{equation*}
    decreases strictly from $\infty$ to $-\infty$ as $p$ increases from $0$ to $1$.
    Thus there is a unique $p_0\in(0,1)$ such that $A_{p_0}(0)=A_{p_0}(1)\eqqcolon\alpha_0$.
    Since $\phi_1\ne\phi_2$ and $\omega\ne0$, \cref{l:const-char} gives $A_{p_0}^{-1}(\alpha_0)=\{0,1\}$.

    Choose $\delta\in(0,1)$ sufficiently small that $D(v)<D(1)$ for all $v\in[0,\delta]$.
    Then $A_{p_0}(\delta)\ne\alpha_0$.
    Since $\Delta$ changes sign at $p_0$, continuity gives a non-empty open interval $J\subset(0,1)$ with endpoint $p_0$ such that $A_p(1)$ lies strictly between $A_p(0)$ and $A_p(\delta)$ for every $p\in J$.
    The intermediate value theorem therefore gives $v_p\in(0,\delta)$ with $A_p(v_p)=A_p(1)$.
    By \cref{l:const-char}, $A_p^{-1}(A_p(1))=\{v_p,1\}$, and $D(v_p)<D(1)$ by the choice of $\delta$.
    Hence \cref{p:disc} gives a jump discontinuity at $A_p(1)$, which lies in the interior of $I_p$.
    Taking $U=\{(p,1-p):p\in J\}$ completes the proof.
\end{proof}
The hypotheses of \cref{t:exc} hold for most Gatzouras--Lalley carpets with two homogeneous columns.
The following result therefore immediately follows.
\begin{corollary}\label{c:excp}
    There exists a Gatzouras--Lalley measure in $\R^2$ generated by an IFS with 3 maps such that the multifractal spectrum is discontinuous, the fine $L^q$-spectrum is non-differentiable, and the multifractal formalism fails.
\end{corollary}
\begin{proof}
    Take $a_1 = 1/2$ and $a_2=1/4$, $\ell_1 = 2$, $\ell_2 = 1$, and $b_1 = b_2 = 1/5$, and let $p < p_0 \approx 0.7170150212$ be sufficiently close to $p_0$, where $p_0$ is the positive root of the quadratic $p^2 +e^x p - e^x = 0$ for $x = 2(\log 2)^2/(\log 5)$.
    For example, a quick numerical calculation using \cref{p:disc} shows that $p = 0.717$ works.
\end{proof}
Using the main result of \cite{zbl:1228.37025}, we can also give an example of a Gatzouras--Lalley measure whose fine $L^q$ spectrum is non-differentiable at $0$.
This is a consequence of the following observation.
\begin{lemma}\label{l:distinct-slopes}
    Let $\Phi$ be a Gatzouras--Lalley IFS with separation of principal projections, two homogeneous columns, and two measures of maximal dimension.
    Then there exists an open set of values $p$ so that $\vartheta_{\bm{p}}$ is non-differentiable at $0$ and the multifractal formalism fails.
\end{lemma}
\begin{proof}
    Let $0 \leq v_1 < v_2 \leq 1$ be the global maximizers of $D(v)$.
    Let $s>0$ be chosen so that $a_1^s + a_2^s = 1$ and let $\bm{z} = (a_1^s, a_2^s)$.
    Then $\CH{\bm{v}}{\bm{z}}=s\chi_1(\bm{v})$ so that
    \begin{equation*}
        A_{\bm{z}}(v)=s+\frac{v\log\ell_1+(1-v)\log\ell_2}{v\log(1/b_1)+(1-v)\log(1/b_2)}.
    \end{equation*}
    Since there are two measures of maximal dimension, $\phi_1 \neq \phi_2$, and therefore $A_{\bm{z}}$ is strictly monotone so $A_{\bm{z}}(v_1)\neq A_{\bm{z}}(v_2)$.
    By continuity, $A_p(v_1) \neq A_p(v_2)$ for all $p$ in a small neighbourhood of $a_1^s$.
    For any such $p$, non-differentiability therefore follows by \cref{l:duality}~\cref{im:subdiff-char}.
    Since there are no other global maximizers, \cref{it:two-col-var,e:subdiff-rel} give $f_{\bm{p}}(\alpha)<D(v_1)=-\vartheta_{\bm{p}}(0)=\vartheta_{\bm{p}}^*(\alpha)$ for every $\alpha$ strictly between $A_p(v_1)$ and $A_p(v_2)$, so the multifractal formalism also fails.
\end{proof}
\begin{example}\label{ex:barral-feng}
    Here, we outline the details from \cite[Example~2.1]{zbl:1228.37025} to give a concrete example of a Gatzouras--Lalley measure whose fine $L^q$-spectrum is non-differentiable at $0$.
    First, set
    \begin{align*}
        A &= \log 3 - \frac{7}{12}\log 2& B &= 3\log 2\\
        U &= (AB)^{1/2}-\frac{B}{2}&V &=\frac{2B + 4(AB)^{1/2}}{4A-B}
    \end{align*}
    We then consider the Gatzouras--Lalley carpet with two homogeneous columns defined by parameters
    \begin{align*}
        \ell_1 &= 1 & \ell_2 &= 150\\
        a_1 &= \exp(-1) & a_2 &= \exp(-(1+V))\\
        b_1 &= b_2 = \left(\frac{\ell_1}{\ell_2}\right)^{V/B}
    \end{align*}
    It is shown in \cite{zbl:1228.37025} that this carpet has two Bernoulli measures of maximal dimension.
    Thus the claim follows by \cref{l:distinct-slopes}.
\end{example}
\begin{figure}[t]
    \begin{center}
        \includegraphics[width = 0.85\textwidth]{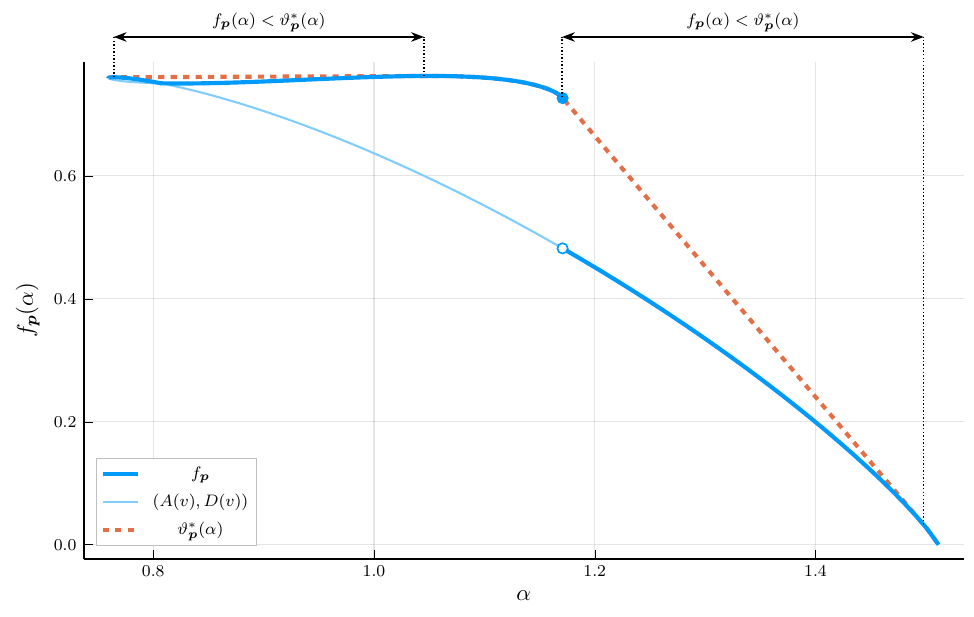}
    \end{center}
    \caption{A plot of the multifractal spectrum of a Gatzouras--Lalley measure with two homogeneous columns.
        The multifractal spectrum is the upper hull of the blue curve, depicted in solid blue.
        The properties are described in \cref{ex:computer}.
    }
    \label{f:exp}
\end{figure}
\begin{example}\label{ex:computer}
    Using a computer search, we were able to find an example with two phase transitions (one for negative $q$ and one for positive $q$), a point where the multifractal spectrum has a jump discontinuity, and a point where the multifractal spectrum is continuous but not differentiable.
    The jump discontinuity corresponds to the endpoint $A_p(0)$ and there is a unique measure of maximal dimension corresponding to $\alpha\approx 1.049925193894$ (between the two regions of non-concavity).
    The exact parameters are:
    \begin{align*}
        \ell_1 &= 1 & \ell_2 &= 12\\
        a_1 &= 0.8484575397963153 & a_2 &=0.03303678167872682\\
        b_1 &= b_2 = 0.03274279864356218\\
        p &= 0.7800520920221643.
    \end{align*}
    See \cref{f:exp} for a depiction of the multifractal spectrum.
\end{example}

\begin{acknowledgements}
    I.K.\ is supported by the Hungarian NRDI Office grants STARTING152587 and K142169 and the János Bolyai Research Scholarship of the Hungarian Academy of Sciences.
    The majority of the work on this project was completed while the third author was a PhD student at the University of St Andrews, where he was supported by EPSRC Grant EP/V520123/1 and the Natural Science and Engineering Research Council (NSERC) via the grant CGS-D/567958-2022.
    A.R.\ also received support from NSERC via the grant PDF/587515-2024 and from the Research Council of Finland via the project \emph{Approximate incidence geometry}, grant no.\ 355453.

    AI tools were used in limited amounts during the final preparation of this paper.
    Most notably, at the direction of the third author, ChatGPT 5.6 Sol was used to write Julia code in order to perform the computer search used to find the example stated in \cref{ex:computer}.
    ChatGPT 5.6 Sol was also used to help with some minor modifications to the TikZ code underlying \cref{f:exp} (the base image was auto-generated by the Julia Plots library).
\end{acknowledgements}

@preprint{arxiv:2312.08974,
  arxiv = {2312.08974},
  author = {Rutar, Alex},
  eprint = {2312.08974},
  eprinttype = {arxiv},
  journal = {Real Anal. Exch.},
  month = {12},
  title = {Multifractal analysis via Lagrange duality},
  year = {2023},
}

@thesis{local:BedfordThesis,
  author = {Bedford, Tim},
  institution = {University of Warwick},
  title = {Crinkly curves, Markov partitions and box dimensions},
  type = {PhD thesis},
  year = {1984},
}

@article{zbl:0757.28011,
  author = {Lalley, Steven P. and Gatzouras, Dimitrios},
  doi = {10.1512/iumj.1992.41.41031},
  eprint = {0757.28011},
  eprinttype = {zbl},
  journal = {Indiana Univ. Math. J.},
  language = {English},
  pages = {533--568},
  title = {Hausdorff and box dimensions of certain self-affine fractals},
  volume = {41},
  year = {1992},
  zbl = {0757.28011},
  zbmath = {00108960},
}

@article{zbl:0955.28004,
  author = {Olsen, Lars},
  doi = {10.2140/pjm.1998.183.143},
  eprint = {0955.28004},
  eprinttype = {zbl},
  journal = {Pac. J. Math.},
  language = {English},
  pages = {143--199},
  title = {Self-affine multifractal {Sierpiński} sponges in {{\(\mathbb{R}^d\)}}},
  volume = {183},
  year = {1998},
  zbl = {0955.28004},
  zbmath = {01180207},
}

@article{zbl:0946.28004,
  addres = {Bristol; London},
  author = {Falconer, Kenneth J.},
  doi = {10.1088/0951-7715/12/4/308},
  eprint = {0946.28004},
  eprinttype = {zbl},
  issue = {4},
  journal = {Nonlinearity},
  language = {English},
  pages = {877--891},
  publisher = {IOP Publishing; London Mathematical Society},
  title = {Generalized dimensions of measures on self-affine sets},
  volume = {12},
  year = {1999},
  zbl = {0946.28004},
  zbmath = {1340053},
}

@book{zbl:1058.90049,
  author = {Boyd, Stephen and Vandenberghe, Lieven},
  eprint = {1058.90049},
  eprinttype = {zbl},
  language = {English},
  pagetotal = {716},
  publisher = {Cambridge University Press},
  title = {Convex optimization},
  year = {2004},
  zbl = {1058.90049},
  zbmath = {2107836},
}

@article{zbl:1091.28005,
  author = {Feng, De-Jun and Wang, Yang},
  doi = {10.1007/s00041-004-4031-4},
  eprint = {1091.28005},
  eprinttype = {zbl},
  journal = {J. Fourier Anal. Appl.},
  language = {English},
  pages = {107--124},
  title = {A class of self-affine sets and self-affine measures},
  volume = {11},
  year = {2005},
  zbl = {1091.28005},
  zbmath = {02171866},
}

@article{zbl:0763.58018,
  author = {Cawley, Robert and Mauldin, R. Daniel},
  doi = {10.1016/0001-8708(92)90064-R},
  eprint = {0763.58018},
  eprinttype = {zbl},
  journal = {Adv. Math.},
  language = {English},
  pages = {196--236},
  title = {Multifractal decompositions of {Moran} fractals},
  volume = {92},
  year = {1992},
  zbl = {0763.58018},
  zbmath = {00036405},
}

@article{zbl:0538.58026,
  author = {Hentschel, H. G. E. and Procaccia, Itamar},
  doi = {10.1016/0167-2789(83)90235-X},
  eprint = {0538.58026},
  eprinttype = {zbl},
  journal = {Physica D},
  language = {English},
  pages = {435--444},
  title = {The infinite number of generalized dimensions of fractals and strange attractors},
  volume = {8},
  year = {1983},
  zbl = {0538.58026},
  zbmath = {03856061},
}

@article{zbl:0539.28003,
  author = {McMullen, Curt},
  doi = {10.1017/S0027763000021085},
  eprint = {0539.28003},
  eprinttype = {zbl},
  journal = {Nagoya Math. J.},
  language = {English},
  pages = {1--9},
  title = {The {Hausdorff} dimension of general {Sierpiński} carpets},
  volume = {96},
  year = {1984},
  zbl = {0539.28003},
  zbmath = {03857419},
}

@article{zbl:0605.58028,
  author = {Ledrappier, F. and Young, L.-S.},
  doi = {10.2307/1971328},
  eprint = {0605.58028},
  eprinttype = {zbl},
  journal = {Ann. Math.},
  language = {English},
  pages = {509--539},
  subtitle = {Part {I}: {Characterization} of measures satisfying {Pesin}'s entropy formula},
  title = {The metric entropy of diffeomorphisms},
  volume = {122},
  year = {1985},
  zbl = {0605.58028},
  zbmath = {03977985},
}

@article{zbl:0664.58022,
  address = {Cambridge},
  author = {Rand, David A.},
  doi = {10.1017/S0143385700005162},
  eprint = {0664.58022},
  eprinttype = {zbl},
  journal = {Ergodic Theory Dyn. Syst.},
  language = {English},
  pages = {527--541},
  publisher = {Cambridge University Press},
  title = {The singularity spectrum {\(f(\alpha)\)} for cookie-cutters},
  volume = {9},
  year = {1989},
  zbl = {0664.58022},
  zbmath = {04086585},
}

@article{zbl:1206.82004,
  author = {Barral, Julien and Mensi, Mounir},
  doi = {10.1017/S0143385706001027},
  eprint = {1206.82004},
  eprinttype = {zbl},
  issue = {5},
  journal = {Ergodic Theory Dyn. Syst.},
  language = {English},
  pages = {1419--1443},
  publisher = {Cambridge University Press, Cambridge},
  title = {Gibbs measures on self-affine Sierpiński carpets and their singularity spectrum},
  volume = {27},
  year = {2007},
  zbl = {1206.82004},
  zbmath = {5209385},
}

@article{zbl:1150.28004,
  arxiv = {0802.0520},
  author = {Barral, Julien and Mensi, Mounir},
  doi = {10.1088/0951-7715/21/10/011},
  eprint = {1150.28004},
  eprinttype = {zbl},
  issue = {10},
  journal = {Nonlinearity},
  language = {English},
  mrnumber = {2439486},
  pages = {2409--2425},
  publisher = {IOP Publishing, Bristol; London Mathematical Society, London},
  title = {Multifractal analysis of {B}irkhoff averages on `self-affine' symbolic spaces},
  volume = {21},
  year = {2008},
  zbl = {1150.28004},
  zbmath = {5355429},
}

@article{zbl:1230.37031,
  author = {Feng, De-Jun and Hu, Huyi},
  doi = {10.1002/cpa.20276},
  eprint = {1230.37031},
  eprinttype = {zbl},
  journal = {Commun. Pure Appl. Math.},
  language = {English},
  pages = {1435--1500},
  title = {Dimension theory of iterated function systems},
  volume = {62},
  year = {2009},
  zbl = {1230.37031},
  zbmath = {05614742},
}

@article{zbl:1184.28009,
  author = {Feng, De-Jun and Lau, Ka-Sing},
  doi = {10.1016/j.matpur.2009.05.009},
  eprint = {1184.28009},
  eprinttype = {zbl},
  journal = {J. Math. Pures Appl.},
  language = {English},
  pages = {407--428},
  title = {Multifractal formalism for self-similar measures with weak separation condition},
  volume = {92},
  year = {2009},
  zbl = {1184.28009},
  zbmath = {05618743},
}

@article{zbl:1206.28012,
  author = {Jordan, Thomas and Rams, Michal},
  doi = {10.1017/S0305004110000472},
  eprint = {1206.28012},
  eprinttype = {zbl},
  journal = {Math. Proc. Camb. Philos. Soc.},
  language = {English},
  pages = {147--156},
  title = {Multifractal analysis for {Bedford}--{McMullen} carpets},
  volume = {150},
  year = {2011},
  zbl = {1206.28012},
  zbmath = {05838021},
}

@article{zbl:1230.37034,
  arxiv = {1008.0301},
  author = {Reeve, Henry W. J.},
  doi = {10.4064/fm212-1-5},
  eprint = {1230.37034},
  eprinttype = {zbl},
  issue = {1},
  journal = {Fundam. Math.},
  language = {English},
  mrnumber = {2771589},
  pages = {71--93},
  publisher = {Polish Academy of Sciences (Polska Akademia Nauk - PAN), Institute of Mathematics (Instytut Matematyczny), Warsaw},
  title = {Multifractal analysis for {B}irkhoff averages on {L}alley--{G}atzouras repellers
},
  volume = {212},
  year = {2011},
  zbl = {1230.37034},
  zbmath = {5860437},
}

@article{zbl:1228.37025,
  author = {Barral, Julien and Feng, De-Jun},
  doi = {10.1088/0951-7715/24/9/010},
  eprint = {1228.37025},
  eprinttype = {zbl},
  journal = {Nonlinearity},
  language = {English},
  pages = {2563--2567},
  title = {Non-uniqueness of ergodic measures with full {Hausdorff} dimension on a {Gatzouras}--{Lalley} carpet},
  volume = {24},
  year = {2011},
  zbl = {1228.37025},
  zbmath = {05956581},
}

@article{zbl:1280.28010,
  arxiv = {1110.6578},
  author = {Barral, Julien and Feng, De-Jun},
  doi = {10.1007/s00220-013-1676-3},
  eprint = {1280.28010},
  eprinttype = {zbl},
  issue = {2},
  journal = {Commun. Math. Phys.},
  language = {English},
  pages = {473--504},
  publisher = {Springer, Berlin/Heidelberg},
  title = {Multifractal formalism for almost all self-affine measures},
  volume = {318},
  year = {2013},
  zbl = {1280.28010},
  zbmath = {6145996},
}

@article{zbl:1371.37012,
  author = {Ledrappier, F. and Young, L.-S.},
  doi = {10.2307/1971329},
  eprint = {1371.37012},
  eprinttype = {zbl},
  journal = {Ann. Math.},
  language = {English},
  pages = {540--574},
  subtitle = {Part {II}: {Relations} between entropy, exponents and dimension},
  title = {The metric entropy of diffeomorphisms},
  volume = {122},
  year = {1985},
  zbl = {1371.37012},
  zbmath = {06781245},
}

@article{zbl:1387.37026,
  author = {Das, Tushar and Simmons, David},
  doi = {10.1007/s00222-017-0725-5},
  eprint = {1387.37026},
  eprinttype = {zbl},
  journal = {Invent. Math.},
  language = {English},
  pages = {85--134},
  title = {The {Hausdorff} and dynamical dimensions of self-affine sponges: a dimension gap result},
  volume = {210},
  year = {2017},
  zbl = {1387.37026},
  zbmath = {06798848},
}

@article{zbl:1379.28008,
  address = {Providence, RI},
  arxiv = {1607.00894},
  author = {Fraser, Jonathan M. and Kempton, Tom},
  doi = {10.1090/proc/13672},
  eprint = {1379.28008},
  eprinttype = {zbl},
  issue = {1},
  journal = {Proc. Am. Math. Soc.},
  language = {English},
  pages = {161--173},
  publisher = {American Mathematical Society},
  title = {On the \(L^q\)-dimensions of measures on Hueter-Lalley type self-affine sets},
  volume = {146},
  year = {2018},
  zbl = {1379.28008},
  zbmath = {6810497},
}

@article{zbl:1407.28002,
  author = {Fraser, Jonathan M. and Yu, Han},
  doi = {10.1512/iumj.2018.67.7509},
  eprint = {1407.28002},
  eprinttype = {zbl},
  journal = {Indiana Univ. Math. J.},
  language = {English},
  pages = {2005--2043},
  title = {Assouad-type spectra for some fractal families},
  volume = {67},
  year = {2018},
  zbl = {1407.28002},
  zbmath = {07019851},
}

@article{zbl:1426.11079,
  author = {Shmerkin, Pablo},
  doi = {10.4007/annals.2019.189.2.1},
  eprint = {1426.11079},
  eprinttype = {zbl},
  journal = {Ann. Math.},
  language = {English},
  pages = {319--391},
  title = {On {Furstenberg}'s intersection conjecture, self-similar measures, and the {{\(L^q\)}} norms of convolutions},
  volume = {189},
  year = {2019},
  zbl = {1426.11079},
  zbmath = {07041748},
}

@article{zbl:0867.28006,
  address = {Cambridge},
  author = {Hueter, Irene and Lalley, Steven P.},
  doi = {10.1017/S0143385700008257},
  eprint = {0867.28006},
  eprinttype = {zbl},
  issue = {1},
  journal = {Ergodic Theory Dyn. Syst.},
  language = {English},
  pages = {77--97},
  publisher = {Cambridge University Press},
  title = {Falconer's formula for the Hausdorff dimension of a self-affine set in \(\mathbb{R}^2\)},
  volume = {15},
  year = {1995},
  zbl = {0867.28006},
  zbmath = {729207},
}

@book{zbl:0819.28004,
  address = {Cambridge},
  author = {Mattila, Pertti},
  eprint = {0819.28004},
  eprinttype = {zbl},
  language = {English},
  publisher = {Univ. Press},
  series = {Camb. Stud. Adv. Math.},
  subtitle = {Fractals and rectifiability},
  title = {Geometry of sets and measures in {Euclidean} spaces},
  volume = {44},
  year = {1995},
  zbl = {0819.28004},
  zbmath = {00739280},
}

@article{zbl:1561.28096,
  author = {Rutar, Alex},
  doi = {10.1017/etds.2022.28},
  eprint = {07682668},
  eprinttype = {zbmath},
  journal = {Ergodic Theory Dyn. Syst.},
  language = {English},
  pages = {2028--2072},
  title = {Geometric and combinatorial properties of self-similar multifractal measures},
  volume = {43},
  year = {2023},
  zbmath = {07682668},
}

@article{zbl:1564.28002,
  address = {University of North Carolina, Chapel Hill, NC},
  arxiv = {1901.01691},
  author = {Feng, De-Jun},
  doi = {10.1215/00127094-2022-0014},
  eprint = {1564.28002},
  eprinttype = {zbl},
  issue = {4},
  journal = {Duke Math. J.},
  language = {English},
  pages = {701--774},
  publisher = {Duke University Press},
  title = {Dimension of invariant measures for affine iterated function systems},
  volume = {172},
  year = {2023},
  zbl = {1564.28002},
  zbmath = {7684350},
}

@article{zbl:1549.37013,
  address = {London},
  arxiv = {2205.01043},
  author = {Kolossváry, István},
  doi = {10.1112/jlms.12767},
  eprint = {1549.37013},
  eprinttype = {zbl},
  journal = {J. Lond. Math. Soc.},
  language = {English},
  pages = {666--701},
  publisher = {London Mathematical Society},
  title = {The \(L^q\) spectrum of self-affine measures on sponges},
  volume = {108},
  year = {2023},
  zbl = {1549.37013},
  zbmath = {07731132},
}

@book{zbl:1543.28001,
  address = {Providence, RI},
  author = {Bárány, Balázs and Simon, Károly and Solomyak, Boris},
  doi = {10.1090/surv/276},
  eprint = {1543.28001},
  eprinttype = {zbl},
  language = {English},
  publisher = {American Mathematical Society},
  series = {Math. Surv. Monogr.},
  title = {Self-similar and self-affine sets and measures},
  volume = {276},
  year = {2023},
  zbl = {1543.28001},
  zbmath = {07759243},
}

@article{zbl:1587.28018,
  arxiv = {2401.07168},
  author = {Banaji, Amlan and Fraser, Jonathan M. and Kolossváry, István and Rutar, Alex},
  doi = {10.1016/j.aim.2025.110707},
  eprint = {08135613},
  eprinttype = {zbmath},
  journal = {Adv. Math.},
  language = {English},
  month = {01},
  pages = {Article ID 110707, 46 p.},
  publisher = {Elsevier (Academic Press), San Diego, CA},
  title = {Assouad spectrum of Gatzouras-Lalley carpets},
  volume = {484},
  year = {2026},
  zbmath = {08135613},
}

@article{zbmath:8179359,
  address = {Berlin},
  arxiv = {2401.13626},
  author = {Batsis, Alex and Käenmäki, Antti and Kempton, Tom},
  doi = {10.1007/s00209-026-03972-2},
  eid = {115},
  eprint = {8179359},
  eprinttype = {zbmath},
  issue = {4},
  journal = {Math. Z.},
  language = {English},
  pagetotal = {25},
  publisher = {Springer},
  title = {Local dimension spectrum for dominated planar self-affine sets},
  volume = {312},
  year = {2026},
  zbmath = {8179359},
}

@article{zbl:0845.28007,
  author = {King, James F.},
  doi = {10.1006/aima.1995.1061},
  eprint = {0845.28007},
  eprinttype = {zbl},
  journal = {Adv. Math.},
  language = {English},
  pages = {1--11},
  title = {The singularity spectrum for general {Sierpiński} carpets},
  volume = {116},
  year = {1995},
  zbl = {0845.28007},
  zbmath = {00838703},
}

@article{zbl:0841.28012,
  author = {Olsen, Lars},
  doi = {10.1006/aima.1995.1066},
  eprint = {0841.28012},
  eprinttype = {zbl},
  journal = {Adv. Math.},
  language = {English},
  pages = {82--196},
  title = {A multifractal formalism},
  volume = {116},
  year = {1995},
  zbl = {0841.28012},
  zbmath = {00838708},
}

@article{zbl:0873.28003,
  author = {Arbeiter, Matthias and Patzschke, Norbert},
  doi = {10.1002/mana.3211810102},
  eprint = {0873.28003},
  eprinttype = {zbl},
  journal = {Math. Nachr.},
  language = {English},
  pages = {5--42},
  title = {Random self-similar multifractals},
  volume = {181},
  year = {1996},
  zbl = {0873.28003},
  zbmath = {00927263},
}

@book{zbl:0869.28003,
  address = {Chichester},
  author = {Falconer, Kenneth},
  eprint = {0869.28003},
  eprinttype = {zbl},
  language = {English},
  publisher = {John Wiley \& Sons},
  title = {Techniques in fractal geometry},
  year = {1997},
  zbl = {0869.28003},
  zbmath = {00981057},
}
\end{document}